\RequirePackage{fix-cm}

\documentclass[11pt,a4paper]{amsart}
\usepackage[utf8]{inputenc}
\usepackage[english]{babel}
\usepackage[T1]{fontenc}
\usepackage{lmodern} 
\rmfamily 
\DeclareFontShape{T1}{lmr}{b}{sc}{<->ssub*cmr/bx/sc}{}
\DeclareFontShape{T1}{lmr}{bx}{sc}{<->ssub*cmr/bx/sc}{}

\usepackage{sidecap}

\usepackage{verbatim}
\usepackage{lmodern}
\usepackage{amsmath}
\usepackage{amssymb} 
\usepackage{amsthm} 
\usepackage{thmtools}
\usepackage{enumitem}
\usepackage{graphicx}
\usepackage[hidelinks]{hyperref}%hidelinks avoid the red boxes around the hyperlinks
\usepackage{mathtools}
\usepackage{indentfirst}
\usepackage{tabularx}
\usepackage{bbm}
\usepackage[capitalise]{cleveref}
\usepackage{stmaryrd}
\usepackage{tabularx}

\usepackage[left=1.2in,top=2cm,right=1.2in,bottom=2cm]{geometry}

\usepackage{todonotes}

\usepackage{tikz} %pictures
\usepackage{tikz-cd} %diagrams
\newcommand{\btk}{\begin{tikzcd}}
\newcommand{\etk}{\end{tikzcd}}
\usetikzlibrary{arrows,calc,decorations.markings}
\usepackage{xparse}

\usepackage{blindtext}

\makeatletter
\newcommand{\@bbify}[1]{
  \ifcsname b#1\endcsname
  \message{WARNING: Overwriting b#1 with blackboard letter!}
  \fi
  \expandafter\edef\csname b#1\endcsname
  {\noexpand\ensuremath{\noexpand\mathbb #1}\noexpand\xspace}}
\newcommand{\@calify}[1]{
  \ifcsname c#1\endcsname
  \message{WARNING: Overwriting c#1 with calligraphic letter!}
  \fi 
  \expandafter\edef\csname c#1\endcsname
  {\noexpand\ensuremath{\noexpand\mathcal #1}\noexpand\xspace}}
\newcommand{\@bfify}[1]{
  \ifcsname bf#1\endcsname
  \message{WARNING: Overwriting c#1 with bold letter!}
  \fi
  \expandafter\edef\csname bf#1\endcsname
  {\noexpand\ensuremath{\noexpand\mathbf #1}\noexpand\xspace}}
\newcounter{@letter}\stepcounter{@letter}
\loop\@bbify{\Alph{@letter}}\@calify{\Alph{@letter}}\@bfify{\Alph{@letter}}
\ifnum\the@letter<26\stepcounter{@letter}\repeat
\makeatother

\newenvironment{tz}{\begin{center}\begin{tikzpicture}}{\end{tikzpicture}\end{center}}
\tikzstyle{d}=[double distance=.3ex]
\tikzstyle{w}=[preaction={draw=white,-,line width=5pt}]

\tikzset{%
node distance=1.5cm, la/.style={scale=0.8}, lasmall/.style={scale=0.75}, over/.style={auto=false,fill=white,inner sep=1.5pt, minimum size=0, outer sep=0},
    symbol/.style={%
        draw=none,
        every to/.append style={%
            edge node={node [sloped, allow upside down, auto=false]{$#1$}}},
            
    }, pro/.style={postaction={decorate,decoration={
        markings,
        mark=at position .5 with {\node at (0,0) {$\bullet$};}
      }},
      inner sep=.9ex,
      },
      prosmall/.style={postaction={decorate,decoration={
        markings,
        mark=at position .5 with {\node at (0,0) {$\scriptstyle \bullet$};}
      }},
      inner sep=.9ex,
      },
  n/.style={double equal sign distance, -implies}, t/.style={double distance=2.5pt, -implies, postaction={draw,-}},
}

\newcommand{\arrowdot}{
\ensuremath{\begin{tikzpicture}
\node (A) at (0,-.4) {};
\node (B) at (.4,-.4) {};
\draw[->, line width=.1ex] (0,-.6) -- (.4,-.6);
\node[shape=circle, fill=black, scale=0.35] (A) at  (.17,-.6) {};
\end{tikzpicture}
}}

\newcommand{\Isodot}{
\ensuremath{\begin{tikzpicture}
\node (A) at (0,-.4) {};
\node (B) at (.4,-.4) {};
\draw[n, line width=.1ex] (0,-.6) to node[above,la]{$\cong$} (.4,-.6);
\node[shape=circle, fill=black, scale=0.35] (A) at  (.17,-.6) {};
\end{tikzpicture}
}}

\newcommand{\Set}{\mathrm{Set}}

\newcommand{\Cof}{\mathrm{Cof}}
\newcommand{\cofM}{\mathrm{Cof}_\cM}
\newcommand{\fibM}{\fib_\cM}
\newcommand{\weM}{\W_\cM}

\newcommand{\trivfibM}{\fibM\cap \weM}
\newcommand{\IM}{\cI_\cM}
\newcommand{\JM}{\cJ_\cM}

\newcommand{\an}{\ensuremath{\mathrm{An}}}

\newcommand{\cof}{\ensuremath{\mathrm{cof}}}

\newcommand{\nfib}{\ensuremath{\mathrm{NFib}}}
\newcommand{\fib}{\ensuremath{\mathrm{Fib}}}

\newcommand{\inj}{\mathrm{inj}}

\newcommand{\fibrant}{\mathrm{fib}}

\newcommand{\Path}{\mathrm{Path}}
\newcommand{\id}{\mathrm{id}}

\newcommand{\op}{\mathrm{op}}
\newcommand{\Dop}{\Delta^{\op}}

\newcommand{\Sq}{\mathrm{Sq}}

\newcommand{\C}{\mathcal{C}}
\newcommand{\I}{\mathcal{I}}
\newcommand{\J}{\mathcal{J}}
\newcommand{\Wf}{\mathcal{W}_f}
\newcommand{\W}{\mathcal{W}}
\newcommand{\dblcat}{\mathrm{DblCat}}

\newcommand{\cat}{\mathrm{Cat}}
\newcommand{\twocat}{2\cat}

\newlist{rome}{enumerate}{7}
\setlist[rome]{label=(\roman*)}

\newtheorem{theorem}{Theorem}[section]
\newtheorem{cor}[theorem]{Corollary}
\newtheorem{prop}[theorem]{Proposition}
\newtheorem{lem}[theorem]{Lemma}
\declaretheorem[name=Theorem,numbered=yes]{theoremA}

\theoremstyle{definition}
\newtheorem{defn}[theorem]{Definition}

\newtheorem{ex}[theorem]{Example}
\newtheorem{notation}[theorem]{Notation}

\theoremstyle{remark}
\newtheorem{rem}[theorem]{Remark}

\crefname{theorem}{Theorem}{Theorems}
\crefname{cor}{Corollary}{Corollaries}
\crefname{prop}{Proposition}{Propositions}
\crefname{lem}{Lemma}{Lemmas}

\crefname{defn}{Definition}{Definitions}
\crefname{terminology}{Terminology}{Terminologies}
\crefname{ex}{Example}{Examples}
\crefname{notation}{Notation}{Notations}
\crefname{descr}{Description}{Descriptions}
\crefname{constr}{Construction}{Constructions}

\crefname{rem}{Remark}{Remarks}

\usepackage{blindtext}

\renewcommand\thepart{\Roman{part}.}

\makeatletter
\renewcommand\part{%
  \par
  \addvspace{4ex}%
  \@afterindenttrue
  \secdef\@part\@spart
}

\def\@part[#1]#2{%
    \ifnum \c@secnumdepth >\m@ne
      \refstepcounter{part}%
      
      \addcontentsline{toc}{section}{\hspace{-.5cm} \bfseries\thepart\hspace{1em}#1}%
    \else
      \addcontentsline{toc}{section}{#1}%
    \fi
    {\parindent \z@ \raggedright
     \interlinepenalty \@M
     \normalfont
     \thispagestyle{empty}
     \ifnum \c@secnumdepth >\m@ne
      \centering\large\textsc{\textbf{\thepart}}\nobreakspace
     \fi
     \centering\large\textsc{\textbf{#2}}%\bfseries
     \par}%
    \nobreak
    \vskip .3cm
    \@afterheading}
\def\@spart#1{%
      \addcontentsline{toc}{part}{#1}%
    {\parindent \z@ \raggedright
     \interlinepenalty \@M
     \normalfont
     \thispagestyle{plain}
     \centering\large\textsc{\textbf{#1}}%\bfseries \MakeUppercase{#1}%
     \par}%
    \nobreak
    \vskip .3cm
    \@afterheading}
\makeatother

\title{Fibrantly-transferred model structures}

\author[L.\ Guetta]{L\'eonard Guetta}
\address{Mathematical Institute, Utrecht University, 3584 CD Utrecht, Nederland}
\email{l.s.guetta@uu.nl}

\author[L.\ Moser]{Lyne Moser}
\address{Fakultät für Mathematik, Universität Regensburg, 93040 Regensburg, Germany}
\email{lyne.moser@ur.de}

\author[M.\ Sarazola]{Maru Sarazola}
\address{School of Mathematics, University of Minnesota, Minneapolis MN, 55415, USA}
\email{maru@umn.edu}

\author[P.\ Verdugo]{Paula Verdugo}
\address{{Max Planck Institute for Mathematics, Bonn NRW, 53111, Germany}}
\email{verdugo@mpim-bonn.mpg.de}

\begin{document}

\begin{abstract}
    We develop new techniques for constructing model structures from a given class of cofibrations, together with a class of fibrant objects and a choice of weak equivalences between them. As a special case, we obtain a more flexible version of the classical right-transfer theorem in the presence of an adjunction. Namely, instead of lifting the classes of fibrations and weak equivalences through the right adjoint, we now only do so between fibrant objects, which allows for a wider class of applications. 
    
    %We use these new tools to exhibit three examples of model structures on $\mathrm{DblCat}$, the category of double categories and double functors. First, we recover the weakly horizontally invariant model structure introduced by Moser--Sarazola--Verdugo in \cite{whi}. Next, we define a new model structure that makes the square functor $\mathrm{Sq}\colon2\cat\to\mathrm{DblCat}$ into a Quillen equivalence, providing a double categorical model for the homotopy theory of 2-categories. Finally, we give an example where a fibrantly-induced model structure exists, but the right-induced one does not.
\end{abstract}

\maketitle

\setcounter{tocdepth}{1}
\tableofcontents

\section{Introduction}

Model categories, introduced by Quillen \cite{Quillen}, are nowadays an ubiquitous tool in algebraic topology and homotopical algebra as they provide an abstract framework to do homotopy theory. A model category $(\cM, \Cof, \fib, \cW)$ consists of a category $\mathcal{M}$ that has limits and colimits, together with three distinguished classes of morphisms in $\cM$ called \emph{cofibrations, fibrations}, and \emph{weak equivalences}. Weak equivalences are the driving force in a model category, and they can encode different flavors of ``sameness'', ranging from the most evident choice of isomorphisms, to a topologically meaningful notion of weak homotopy equivalence, to the quasi-isomorphisms one encounters in algebra. In turn, the additional classes of fibrations and cofibrations facilitate the construction of ``path'' and ``cylinder'' objects, as well as the computation of homotopy limits and colimits. Moreover, the available tools in a model category allow us to compare homotopy theories through derived functors and derived equivalences.

The robustness of this type of structure comes with a drawback: in practice, it is often very hard to prove that three given classes of morphisms in a category satisfy the requirements of a model structure. To this end, there are several results in the literature that provide techniques for constructing model structures; see for instance \cite{transfer1, beke, cisinski, mixedMS, olschok, stanc, HKRS}. A particularly advantageous situation happens when the category $\cC$ on which we wish to construct a model structure is related to a known model structure $\cM$ through an adjunction
\begin{tz}
\node[](1) {$\C$}; 
\node[right of=1,xshift=1cm](2) {$\cM.$}; 

\draw[->] ($(1.east)-(0,5pt)$) to node[below,la]{$R$} ($(2.west)-(0,5pt)$);
\draw[->] ($(2.west)+(0,5pt)$) to node[above,la]{$L$} ($(1.east)+(0,5pt)$);

\node[la] at ($(1.east)!0.5!(2.west)$) {$\bot$};
\end{tz}
The standard way to proceed, which goes back at least to Crans \cite{transfer1}, is to \emph{right-transfer} a model structure on~$\cC$ from the one on~$\cM$ along the right adjoint $R$. If successful, this produces a model structure in which a morphism $f$ in $\cC$ is a weak equivalence (resp.\ fibration) if and only if $Rf$ is a weak equivalence (resp.\ fibration) in $\cM$. Hess\textendash K\c{e}dziorek\textendash Riehl\textendash Shipley give in \cite{HKRS} a streamlined list of conditions one can check to build such a model structure, which are inspired by Quillen's Path Object Argument \cite{Quillen}.

However, in some settings, attempting to right-transfer a model structure proves too ambitious. An example can be found by considering $\dblcat$, the category of double categories and double functors. Double categories are 2-dimensional structures having objects, and both horizontal and vertical morphisms, as well as 2-dimensional morphisms called \emph{squares}. Introduced by Ehresmann in \cite{ehresmann}, they have been the focus of much work in the past decades; see for instance \cite{grandispare1,dawsonpare,holonomy,grandispare2,fiore,kock, shulman1,shulman,garner,grandispare3}, among several others. As part of their program to construct model structures on $\dblcat$, Fiore\textendash Paoli\textendash Pronk \cite{FPP} study their relation to the canonical model structure on $\cat$\textemdash the category of (small) categories and functors\textemdash through the horizontal nerve $N^h\colon \dblcat\to \cat^{\Dop}$, which is a right adjoint. Unfortunately, they find that it is not possible to right-transfer a model structure on $\dblcat$ along $N^h$ from the Reedy model structure on $\cat^{\Dop}$; see \cite[Theorem 7.22]{FPP}.

A similar situation happens yet again when we consider the category $\mathrm{Sp}^\Sigma$ of symmetric spectra, which provides one of the most general and flexible frameworks for stable homotopy theory, as a symmetric monoidal model structure. Unfortunately, the classical constructions of the stable model structure for symmetric spectra due to Hovey--Shipley--Smith \cite{HSS} and Mandell--May--Schwede--Shipley \cite{MMSS} are far from straightforward. A major obstruction stems from the fact that one cannot right-transfer this model structure through the forgetful functor $U\colon \mathrm{Sp}^\Sigma\to \mathrm{Sp}_{\mathrm{st}}^{\mathbb{N}}$ to the well-understood stable model structure on sequential spectra of Bousfield--Friedlander \cite{BF}.

The true reason behind these failures is not merely technical, but instead comes from a meaningful fact in homotopy theory: in many settings, one should not expect morphisms between all objects to exhibit a good behavior, but rather only between fibrant ones (i.e., those objects $X\in\cC$ for which the unique morphism to the terminal object $X\to 1$ is a fibration). Fibrant objects play a key role in a model category, and they often consist of the ``well-behaved'' objects in the theory of interest. For instance, in the injective model structure on non-negatively graded cochain complexes, the fibrant objects are the degreewise injective complexes; in the classical model structure for topological spaces, all objects are fibrant (modeling $\infty$-groupoids); in the Joyal model structure on simplicial sets \cite{JoyalVolumeII}, the fibrant objects are the quasi-categories (modeling $\infty$-categories); and in the stable model structure on symmetric spectra, the fibrant objects are the $\Omega$-spectra. %Moreover, the latter two model structures are well-known examples where we do not have convenient characterization of all weak equivalences and fibrations, but rather only those between quasi-categories; namely, these are the equivalences of $\infty$-categories and the isofibrations; see for instance \cite[Section 6]{joyalMS}. 
Moreover, the latter two model structures are well-known examples where the intuitive and convenient characterizations one would hope for do not hold for all weak equivalences and fibrations, but rather only for those between fibrant objects. For quasi-categories, these are the equivalences of $\infty$-categories and the isofibrations; see for instance \cite[Section 6]{joyalMS}. For $\Omega$-spectra, these are the $\pi_*$-isomorphisms and the levelwise fibrations; see \cite{HSS}.

However, this need not be an impediment in finding a model structure that captures the intended homotopy theory. Indeed, together with a class of trivial fibrations, the fibrant
objects and weak equivalences between them fully determine the homotopy theory\footnote{The homotopy category or $\infty$-category associated to a model category can be constructed by first restricting to the cofibrant-fibrant objects, and then inverting the weak equivalences between them; see
\cite{dwyerspalinski,dwyerkan}}, and so it should be possible to work in a setting that uses only this data, as opposed to over-determining all weak equivalences. With this in mind, one would like to have a technique to \textit{fibrantly-transfer} a model structure through a right adjoint as above, where now the weak equivalences and fibrations are only transferred \textit{between fibrant objects}. This is the content of Theorem 3.5, outlined below.

\begin{theoremA}\label{thm1intro}
Let $(\cM,\cofM,\fibM,\weM)$ be a combinatorial model category, and let $\C$ be a locally presentable category. Suppose that we have an adjunction 
\begin{tz}
\node[](1) {$\C$}; 
\node[right of=1,xshift=1cm](2) {$\cM$}; 

\draw[->] ($(1.east)-(0,5pt)$) to node[below,la]{$R$} ($(2.west)-(0,5pt)$);
\draw[->] ($(2.west)+(0,5pt)$) to node[above,la]{$L$} ($(1.east)+(0,5pt)$);

\node[la] at ($(1.east)!0.5!(2.west)$) {$\bot$};
\end{tz}
and that the following properties are satisfied:
\begin{enumerate}[label=(\arabic*)]
\item  every morphism in $R^{-1}(\trivfibM)$ has a fibrant replacement in $R^{-1}(\weM)$,
    \item  for every object $X$ such that $RX$ is fibrant, there is a factorization of the diagonal morphism
    $ X\xrightarrow{w} P X\xrightarrow{p} X\times X$
    with $w\in R^{-1}(\weM)$ and $p\in R^{-1}(\fibM)$.
\end{enumerate}
Then, there exists a combinatorial model structure on $\C$ in which an object $X$ is fibrant if and only if $RX$ is fibrant in $\cM$, and a morphism $f$ is a trivial fibration if and only if $Rf$ is so in $\cM$. Moreover, a morphism $f$ in $\cC$ between fibrant objects is a weak equivalence (resp.\ fibration) if and only if $Rf$ is so in $\cM$.
\end{theoremA}

Whenever the classical right-transfer result holds, then so does \cref{thm1intro}, and both produce the same model structure on the category $\cC$. However, this new method has the flexibility required to avoid the issues preventing us from using the classical right-transfer results in certain settings. 

For instance, our result can now be applied to the nerve functor $N^h\colon \dblcat\to \cat^{\Dop}$ to obtain a fibrantly-transferred model structure on $\dblcat$ from the Reedy model structure on $\cat^{\Dop}$, which is the aim of \cref{application:cat}. This gives a first example where the right-transferred model structure does not exist, but the fibrantly-transferred one does. In a similar vein, in forthcoming work we use \cref{thm1intro} to define a model structure on $\cat(n\cat)$\textemdash the category of categories internal to $n$-categories\textemdash related to the canonical model structure on $n\cat$.

A second example appears in recent work of Malkiewich and Sarazola \cite{symsp}, where they consider the forgetful functor $U\colon \mathrm{Sp}^\Sigma\to \mathrm{Sp}_{\mathrm{st}}^{\mathbb{N}}$ from the category of symmetric spectra, to the category of sequential spectra endowed with the stable model structure. They use \cref{thm1intro} to obtain the stable model structure on $\mathrm{Sp}^\Sigma$; notably, this makes it possible to completely avoid introducing the much more technical notion of stable equivalences between non-fibrant objects, and only rely on well-known facts about $\pi_*$-isomorphisms in~$\mathrm{Sp}^{\mathbb{N}}$.

\cref{thm1intro} is in fact an instance of a more general technique for constructing model structures from a given class of cofibrations, and of fibrant objects together with a choice of weak equivalences between these, details of which can be found in \cref{thm:main}.

\begin{theoremA}\label{thm2intro}
Let $\C$ be a locally presentable category and $\I$ be a set of morphisms in~$\C$. Suppose in addition that we have a class of ``naive fibrant objects'' in $\C$ and a class $\Wf$ of morphisms in $\C$ between them satisfying certain technical conditions.

Then there exists a combinatorial model structure on $\C$ with cofibrations generated by the set $\cI$, fibrant objects given by the naive fibrant objects, and such that the class of weak equivalences between fibrant objects is precisely $\Wf$.
\end{theoremA}

This result is similar in flavor to Smith's theorem \cite{beke} and to a result by Stanculescu \cite{stanc}, but unlike these, it restricts the conditions to be verified to the morphisms between fibrant objects which is where we assume the user has the most control. In practice, this restriction significantly reduces the difficulty for the user, and we expect this theorem to have a wide range of applications.

 Finally, in recent work \cite{dblcatequivs}, the second, third, and fourth named authors show how \cref{thm2intro} can be used to construct a plethora of model structures on $\dblcat$ whose homotopy theories encode a range of 2-dimensional structures, such as 2-categories and 2-groupoids. These recover the model structure on $\dblcat$ for weakly horizontally invariant double categories established by Moser--Sarazola--Verdugo in \cite{whi}, as well as the gregarious model structure of Campbell \cite{Camp} and the model structure for  equipments of Verdugo \cite{pauthesis}. Moreover, we  construct a new model structure on $\dblcat$ 
making the square functor $\Sq\colon\twocat\to\dblcat$ into a Quillen equivalence, shedding light on an $\infty$-categorical question by Gaitsgory\textendash Rozenblyum \cite[Chapter 10, Theorem 5.2.3]{roz} which was recently proved by Abell\'an \cite{Abellan}.

Since the first version of this paper was made available, even more examples of \cref{thm2intro} have come to light (see \cite{MoserNuiten,symsp,grothMS}); we give an overview of these in \cref{section:newexamples}. 
We anticipate that Theorems A and B will continue to be of use in many new applications.

\subsection*{Acknowledgments}

The authors would like to thank Yonatan Harpaz for suggesting one of the applications included in this paper. We are also grateful to Emily Riehl, Martina Rovelli, Viktoriya Ozornova, and Dominic Verity for helpful discussions related to the subject of this paper.

This material is based upon work supported by the National Science Foundation under Grant No.\ DMS-1928930 while the last three named authors participated in a program supported by the MSRI (now the Simons Laufer Mathematical Sciences Institute). The program was held in the summer of 2022 in partnership with the Universidad Nacional Aut\'onoma de M\'exico.

During the realization of this work, the first and second named authors were supported by the Max Planck Institute for Mathematics, and additionally, the second named author was a member of the Collaborative Research Centre ``SFB 1085: Higher Invariants'' funded by the Deutsche Forschungsgemeinschaft (DFG). The third named author was partially supported by a National Science Foundation grant DMS-2506116, and the last named author was partially supported while at Johns Hopkins by a National Science Foundation grant DMS-1652600, and also by an international Macquarie University Research Excellence Scholarship.

\section{Fibrantly constructed model structures}\label{section:abstract}

The main result in this paper is designed to be used in a setting where one has a locally presentable category $\C$, and classes of desired cofibrations and fibrant objects in mind, together with classes of weak equivalences and fibrations between these fibrant objects, and wishes to extend this to a model structure on the category $\C$.

In order to achieve this, we will require the existence of an auxiliary weak factorization system $(\an,\nfib)$ generated by a set $\J$, that serves a similar role as the \emph{anodyne extensions} and \emph{naive fibrations} of \cite{cisinski}, from which we borrow this terminology. As their name suggests, anodyne extensions will be some, but not all, of the trivial cofibrations in our model structure. However, anodyne extensions play a crucial role in constructing the class of weak equivalences of the model structure. 

To explain this, let us briefly recall how any weak  factorization system $(\an,\nfib)$ provides us with a notion of ``fibrant objects'' and of ``fibrant replacements'' relative to the factorization system, which in the context of our paper we will call \emph{naive fibrant objects} and  \emph{naive fibrant replacements}.

\begin{defn}
An object $X\in \C$ is \emph{naive fibrant} if the unique morphism $X\to 1$ to the terminal object is a naive fibration. 
\end{defn}

\begin{defn}
Given an object $X\in\C$, a \emph{naive fibrant replacement} of $X$ is an anodyne extension
\[\iota \colon X \to X'\]
such that $X'$ is naive fibrant.

Similarly, given a morphism $f \colon X\to Y$ in $\C$, a naive fibrant replacement of  $f$ is a commutative square
\begin{tz}
\node[](1) {$X$}; 
\node[right of=1](2) {$Y$}; 
\node[below of=1](1') {$X'$}; 
\node[below of=2](2') {$Y'$}; 

\draw[->] (1) to node[above,la]{$f$} (2); 
\draw[->] (1) to node[left,la]{$\iota_X$} (1'); 
\draw[->] (2) to node[right,la]{$\iota_Y$} (2'); 
\draw[->] (1') to node[below,la]{$f'$} (2');
\end{tz}
where $X'$, $Y'$ are naive fibrant, and $\iota_X$, $\iota_Y$ are anodyne extensions.
\end{defn}

\begin{rem} \label{rem:naivefibrep}
Note that naive fibrant replacements always exist, since we can use the  factorization system $(\an,\nfib)$ to factor the canonical morphism $X \to 1$ as an anodyne extension followed by a naive fibration.

Similarly, for any morphism $f\colon X\to Y$, we can always find a naive fibrant replacement of $f$ by taking naive fibrant replacements of $X$ and $Y$ and using the lifting property of anodyne extensions against naive fibrant objects as follows.
\begin{tz}
\node[](1) {$X$}; 
\node[right of=1](2) {$Y$}; 
\node[below of=1](1') {$X'$}; 
\node[right of=2](3) {$Y'$}; 

\draw[->] (1) to node[above,la]{$f$} (2); 
\draw[->] (2) to node[above,la]{$\iota_Y$} (3); 
\draw[->] (1) to node[left,la]{$\iota_X$} (1'); 
\draw[->,dashed] (1') to node[below,la]{$f'$} (3);
\end{tz}

More abstractly, a result by Bourke and Garner guarantees that any set $\J$ of morphisms in a locally presentable category $\C$ generates an accessible algebraic weak factorization system whose underlying ordinary weak factorization system is the usual weak factorization system
cofibrantly generated by $\J$ (see \cite[Proposition 16]{garner}). In particular, this implies the existence of a naive fibrant replacement \emph{functor} $(-)^{\fibrant}\colon \C\to \C$ which is \emph{accessible}.
\end{rem}

Now, if in addition to the factorization system $(\an, \nfib)$ we are given a class $\Wf$ of morphisms between naive fibrant objects containing the isomorphisms, which we interpret as weak equivalences between fibrant objects, we can  use the above definitions to construct a new class $\W$, which will play the role of our weak equivalences in $\C$.

\begin{defn}\label{def:we}
A morphism $f \colon X \to Y$ in $\C$ is a \emph{weak equivalence} if there exists a naive fibrant replacement of $f$
\begin{tz}
\node[](1) {$X$}; 
\node[right of=1](2) {$Y$}; 
\node[below of=1](1') {$X'$}; 
\node[below of=2](2') {$Y'$}; 

\draw[->] (1) to node[above,la]{$f$} (2); 
\draw[->] (1) to node[left,la]{$\iota_X$} (1'); 
\draw[->] (2) to node[right,la]{$\iota_Y$} (2'); 
\draw[->] (1') to node[below,la]{$f'$} (2');
\end{tz}
such that $f'$ is in $\Wf$. We denote by $\W$ the class of weak equivalences.
\end{defn}

Directly from the definition, we can observe the following.

\begin{lem}\label{ex:anodynewe}
Every anodyne extension is a weak equivalence.
\end{lem}
\begin{proof}
If $f \colon X\to Y$ is an anodyne extension, and $ \iota_Y \colon Y \to Y'$ is a naive fibrant replacement of $Y$, then the commutative square
\begin{tz}
\node[](1) {$X$}; 
\node[right of=1](2) {$Y$}; 
\node[below of=1](1') {$Y'$}; 
\node[below of=2](2') {$Y'$}; 

\draw[->] (1) to node[above,la]{$f$} (2); 
\draw[->] (1) to node[left,la]{$\iota_Y f$} (1'); 
\draw[->] (2) to node[right,la]{$\iota_Y$} (2'); 
\draw[->] (1') to node[below,la]{$\id_{Y'}$} (2');
\end{tz}
 is a naive fibrant replacement of $f$. As $\id_{Y'}\in\Wf$, by definition we have that $f\in\W$.
\end{proof}

\begin{rem}
The fact that anodyne extensions are weak equivalences, together with the foresight that naive fibrant objects will agree with the fibrant objects, imply that these naive fibrant replacements will truly be fibrant replacements in the proposed model structure. Moreover, note that in any model structure, a morphism is a weak equivalence if and only if any of its fibrant replacements is also a weak equivalence. Recalling that we intuitively regard the class $\Wf$ as the weak equivalences between naive fibrant objects, these observations justify our definition of the class of weak equivalences $\W$.
\end{rem}

Having introduced the main characters, we are now equipped to state our theorem, whose proof  will occupy the remainder of this section. First, let us recall the following standard terminology. 

\begin{notation}\label{notation}
Let $\mathcal{I}$ be a set of morphisms in a cocomplete category $\cC$. A morphism in~$\cC$ is
\begin{rome}
   \item \textbf{$\cI$-injective} if it has the right lifting property with respect to every morphism in~$\cI$; the class of all such morphisms is denoted $\inj(\I)$,
    \item an \textbf{$\cI$-cofibration} if it has the left lifting property with respect to every $\cI$-injective morphism; the class of all such morphisms is denoted $\cof(\I)$.
\end{rome}

Recall that in a locally presentable category $\cC$, the pair $(\cof(\I),\inj(\I))$ forms a weak factorization system for any set $\cI$ of morphisms in $\cC$.

Additionally, if $\mathcal{D}$ is any class of morphisms in $\C$, we use $\mathcal{D}^\boxslash$ to denote the class of morphisms in $\C$ that have the right lifting property with respect to every morphism in $\mathcal{D}$. 
\end{notation}

\begin{theorem}\label{thm:main}
Let $\C$ be a locally presentable category, $\I$ be a set of morphisms in $\C$, and $(\an,\nfib)$ be a weak factorization system in $\C$ generated by a set such that $\an\subseteq\cof(\I)$. Suppose in addition that we have a class $\Wf$ of morphisms in $\C$ between naive fibrant objects such that
\begin{enumerate}[label=(\arabic*)]
    \item\label{ax:trivfib} $\inj(\I)\subseteq\W$, where $\W$ is described in \cref{def:we},
    \item\label{2of6Wf} $\Wf$ has 2-out-of-6, 
    \item \label{accessibility} there exists a class $\overline{\cW}$ of morphisms such that $\Wf$ is the restriction of $\overline{\cW}$ to the morphisms between naive fibrant objects and $\overline{\cW}$ considered as a full subcategory of $\C^{\mathbbm{2}}$ is accessible,
    \item\label{anwe} the morphisms in $\an$ between naive fibrant objects are in $\Wf$,
    \item\label{fibwe} $\nfib\cap\Wf\subseteq\inj(\I)$.
\end{enumerate}
Then there exists a combinatorial model structure on $\C$ with cofibrations given by the $\I$-cofibrations, fibrant objects given by the naive fibrant objects, and weak equivalences given by the morphisms in $\W$. Furthermore, weak equivalences (resp.\ fibrations) between fibrant objects are precisely the morphisms in $\Wf$ (resp.\ $\nfib$).
\end{theorem}

\begin{rem}
Recall that in a model structure whose class of cofibrations is given by $\cof(\I)$, the class of weak equivalences $\cW$ must satisfy 2-out-of-3, contain the trivial fibrations $\inj(\I)$, and be such that the trivial cofibrations $\cof(\I)\cap\cW$ are stable under pushout and transfinite composition. One can then verify that, under the hypotheses of \cref{thm:main}, the class $\cW$ of \cref{def:we} is the smallest class of morphisms of $\cC$ that both satisfies these requirements and contains all anodyne extensions. Since we want anodyne extensions to be trivial cofibrations, our choice of weak equivalences is optimal in this sense. Similar minimality results are studied in \cite{rosicky2003left}, \cite{HenryMinimal}.
\end{rem}

\begin{rem}
The above theorem can be compared to a result by Stanculescu \cite[Proposition 1.3]{stanc}, which also describes a method for constructing a model structure from a prescribed class of fibrant objects. The crucial difference between the two results lies in the requirements regarding the weak equivalences. Stanculescu assumes the existence of a  well-behaved class of weak equivalences. In contrast, we only require the user to provide a well-behaved class of weak equivalences between fibrant objects, and the bulk of the technical work of this paper is devoted to showing that these properties can be extended to a larger class of weak equivalences. We foresee that, in practice, this will make \cref{thm:main} friendlier to a user who only has control over the morphisms between fibrant objects; in \cref{section:newexamples} we direct the reader to several examples of this scenario.%. For example, in \cref{section:recoverwhi,section:squareMS}, we provide applications of this result to two cases where we can only explicitly describe the weak equivalences between fibrant objects. 
\end{rem}

In order to prove \cref{thm:main}, we will rely on the following result. Attributed to Smith, this was first published in \cite[Theorem 1.7]{beke}, although we use the reformulation of \cite[Proposition 2.2]{barwick}.

\begin{theorem}[Smith]
Let $\C$ be a locally presentable category, $\W$ be a class of morphisms in $\C$, and $\I$ be a set of morphisms in $\C$. Write $\fib\coloneqq (\cof(\I)\cap\W)^{\boxslash}$. Suppose that the following assumptions hold: 
\begin{enumerate}[label=(\Roman*)]
    \item \label{Smith:trivfib} $\inj(\I)\subseteq \W$,
    \item \label{Smith:2of3} $\W$ satisfies the $2$-out-of-$3$ property,
    \item \label{Smith:accessibility} $\W$ considered as a full subcategory of $\C^{\mathbbm{2}}$ is accessible, 
    \item \label{Smith:trivcof} $\cof(\I)\cap\W$ is closed under pushout and transfinite composition.
\end{enumerate}
Then, there exists a combinatorial model structure on $\C$ whose classes of cofibrations, fibrations, and weak equivalences are given by $(\cof(\I),\fib,\W)$.
\end{theorem}

Let us now study each of the requirements of Smith's theorem. First, note that condition \ref{Smith:trivfib} in Smith's theorem and condition \ref{ax:trivfib} in \cref{thm:main} are identical.  

\subsection{2-out-of-3 property for \texorpdfstring{$\W$}{W}}

Proving that our class of weak equivalences $\W$ given in \cref{def:we} satisfies 2-out-of-3 requires more work, and we build up to this result through several technical lemmas.

\begin{lem}\label{lemma:prelim2out3}
Given a commutative triangle in $\C$
\begin{tz}
\node[](1) {$X$}; 
\node[right of=1,xshift=.2cm](2) {$Y$}; 
\node[below of=2](2') {$Z$}; 

\draw[->] (1) to node[above,la]{$f$} (2); 
\draw[->] (2) to node[right,la]{$g$} (2'); 
\draw[->] (1) to node[below,la,pos=0.4,xshift=-2pt]{$h$} (2');
\end{tz}
where $f$, $h$ are anodyne extensions, and $Y$, $Z$ are naive fibrant objects, we have that the morphism $g$ is in $\Wf$.
\end{lem}

\begin{proof}
Consider the pushout of $f$ along $h$, together with the unique induced morphism $k\colon P\to Z$ as depicted below.
\begin{tz}
\node[](1) {$X$}; 
\node[right of=1](2) {$Y$}; 
\node[below of=1](1') {$Z$}; 
\node[below of=2](2') {$P$}; 

\draw[->] (1) to node[above,la]{$f$} (2); 
\draw[->] (1) to node[left,la]{$h$} (1'); 
\draw[->] (2) to node[right,la]{$h'$} (2'); 
\draw[->] (1') to node[below,la]{$f'$} (2');

\node[below right of=2',xshift=.5cm](5) {$Z$};
\draw[->,bend left] (2) to node[right,la]{$g$} (5); 
\draw[->,bend right] (1') to node[below,la]{$\id_Z$} (5); 
\draw[->,dashed] (2') to node[above,la,pos=0.4,xshift=2pt]{$k$} (5);
\node at ($(2')-(.3cm,-.3cm)$) {$\ulcorner$};
\end{tz}
We can factor the morphism $k$ as $P \xrightarrow{i} Q \xrightarrow{p} Z$ with $i$ an anodyne extension and $p$ a naive fibration. Note that since $Z$ is naive fibrant, so is $Q$.

Now, since $\an$ is the left class in a factorization system, it is stable under pushouts, which guarantees that $f'\in\an$. Then $i f' \colon Z \to Q$ is an anodyne extension between naive fibrant objects, and hence by condition \ref{anwe} we have that $i f'\in\Wf$. By a similar argument, we see that $i  h'\in\Wf$. 

As $\id_Z=kf'=pif'$, it follows from 2-out-of-3 for $\Wf$ that $p\in \Wf$, and then that $g=k  h'=pih'\in \Wf$.
\end{proof}

\begin{lem} \label{lem:univfibrepl}
Let $f\colon X\to Y$ be a morphism in $\C$. Then, there exists a weakly initial naive fibrant replacement as depicted below left; namely, for any other naive fibrant replacement as depicted below right,
\begin{tz}
\node[](1) {$X$}; 
\node[right of=1](2) {$Y$}; 
\node[below of=1](1') {$\hat{X}$}; 
\node[below of=2](2') {$\hat{Y}$}; 

\draw[->] (1) to node[above,la]{$f$} (2); 
\draw[->] (1) to node[left,la]{$u_X$} (1'); 
\draw[->] (2) to node[right,la]{$u_Y$} (2'); 
\draw[->] (1') to node[below,la]{$\hat{f}$} (2');

\node[right of=2,xshift=2cm](1) {$X$}; 
\node[right of=1](2) {$Y$}; 
\node[below of=1](1') {$\overline{X}$}; 
\node[below of=2](2') {$\overline{Y}$}; 

\draw[->] (1) to node[above,la]{$f$} (2); 
\draw[->] (1) to node[left,la]{$j_X$} (1'); 
\draw[->] (2) to node[right,la]{$j_Y$} (2'); 
\draw[->] (1') to node[below,la]{$\overline{f}$} (2');
\end{tz}
there exist morphisms $r\colon\hat{X}\to \overline{X}$ and $s\colon\hat{Y}\to \overline{Y}$ making the following diagram commute. 
\begin{tz}[node distance=2cm]
\node[](1) {$X$}; 
\node[right of=1](2) {$Y$}; 
\node[below of=1](1') {$\hat{X}$}; 
\node[below of=2](2') {$\hat{Y}$}; 

\draw[->] (1) to node[above,la]{$f$} (2); 
\draw[->] (1) to node[left,la]{$u_X$} (1'); 
\draw[->] (2) to node[left,la]{$u_Y$} (2'); 
\draw[->] (1') to node[above,la,pos=0.28]{$\hat{f}$} (2');

\node[below right of=1',xshift=.7cm,yshift=.4cm](1'') {$\overline{X}$}; 
\node[below right of=2',xshift=.7cm,yshift=.4cm](2'') {$\overline{Y}$}; 
\draw[->] (1'') to node[below,la]{$\overline{f}$} (2'');
\draw[->,w] (1) to node[right,la,pos=0.3,yshift=2pt]{$j_X$} (1''); 
\draw[->] (1') to node[left,la,xshift=-10pt,pos=0.6]{$r$} (1''); 
\draw[->] (2) to node[right,la,pos=0.3,yshift=2pt,xshift=2pt]{$j_Y$} (2''); 
\draw[->] (2') to node[left,la,xshift=-10pt,pos=0.6]{$s$} (2''); 
\end{tz}
Moreover, we have that $\hat{f}\in \Wf$ if and only if $\overline{f}\in \Wf$. 
\end{lem}

\begin{proof}
To construct the weakly initial naive fibrant replacement, we consider a naive fibrant replacement $u_X\colon X\to \hat{X}$ and take the pushout  depicted below left.
\begin{tz}
\node[](1) {$X$}; 
\node[right of=1,xshift=.3cm](2) {$Y$}; 
\node[below of=1](1') {$\hat{X}$}; 
\node[below of=2](2') {$\hat{X}\sqcup_X Y$}; 
\node at ($(2')-(.4cm,-.4cm)$) {$\ulcorner$};

\draw[->] (1) to node[above,la]{$f$} (2); 
\draw[->] (1) to node[left,la]{$u_X$} (1'); 
\draw[->] (2) to node[right,la]{$i$} (2'); 
\draw[->] (1') to node[below,la]{$k$} (2');

\node[right of=2,xshift=2cm](1) {$X$}; 
\node[right of=1,xshift=.3cm](2) {$Y$}; 
\node[below of=1](1') {$\hat{X}$}; 
\node[below of=2](2') {$\hat{Y}$}; 

\draw[->] (1) to node[above,la]{$f$} (2); 
\draw[->] (1) to node[left,la]{$u_X$} (1'); 
\draw[->] (2) to node[right,la]{$u_Y\coloneqq ji$} (2'); 
\draw[->] (1') to node[below,la]{$\hat{f}\coloneqq jk$} (2');
\end{tz}
Note that $i\in\an$ since $\an$ is closed under pushout, but $\hat{X}\sqcup_X Y$ may not be naive fibrant. To correct this, let $j\colon \hat{X}\sqcup_X Y\to \hat{Y}$ be a fibrant replacement, and define the weakly initial naive fibrant replacement to be given by the diagram above right.

We now construct the morphisms $r$ and $s$. Since $u_X\in\an$ and $\overline{X}$ is naive fibrant, we have a lift $r$ in the following diagram. 
\begin{tz}
\node[](1) {$X$}; 
\node[right of=1](2) {$\overline{X}$}; 
\node[below of=1](1') {$\hat{X}$}; 

\draw[->] (1) to node[above,la]{$j_X$} (2); 
\draw[->] (1) to node[left,la]{$u_X$} (1');
\draw[->,dashed] (1') to node[right,la,pos=0.4,yshift=-2pt]{$r$} (2);
\end{tz}
Next, by the universal property of the pushout, we know that there exists a morphism $\ell$ making the diagram below left commute.
\begin{tz}
\node[](1) {$X$}; 
\node[right of=1,xshift=.5cm](2) {$Y$}; 
\node[below of=1](1') {$\hat{X}$}; 
\node[below of=2](2') {$\hat{X}\sqcup_X Y$}; 

\draw[->] (1) to node[above,la]{$f$} (2); 
\draw[->] (1) to node[left,la]{$u_X$} (1'); 
\draw[->] (2) to node[right,la]{$i$} (2'); 
\draw[->] (1') to node[above,la]{$k$} (2');

\node[below right of=1',xshift=.5cm](4) {$\overline{X}$};
\node[below right of=2',xshift=.5cm](5) {$\overline{Y}$};
\draw[->,bend left] (2) to node[right,la]{$j_Y$} (5); 
\draw[->] (1') to node[below,la,pos=0.4,xshift=-2pt]{$r$} (4); 
\draw[->] (4) to node[below,la]{$\overline{f}$} (5);
\draw[->,dashed] (2') to node[above,la,pos=0.4,xshift=2pt]{$\ell$} (5);
\node at ($(2')-(.3cm,-.3cm)$) {$\ulcorner$};

\node[right of=2,xshift=3.5cm](1) {$\hat{X}\sqcup_X Y$}; 
\node[right of=1,xshift=.5cm](2) {$\hat{Y}$}; 
\node[below of=1](1') {$\overline{Y}$}; 

\draw[->] (1) to node[above,la]{$\ell$} (2); 
\draw[->] (1) to node[left,la]{$j$} (1');
\draw[->,dashed] (1') to node[right,la,pos=0.4,yshift=-2pt]{$s$} (2);
\end{tz}
Recalling that $j\colon \hat{X}\sqcup_X Y\to \hat{Y}$ is in $\an$ and that $\hat{Y}$ is naive fibrant, we get a lift $s$ in the diagram above right. It then follows that $r$ and $s$ satisfies the desired relations: we have that \[ su_Y=sji=\ell i=j_Y \quad \text{and} \quad s\hat{f}=sjk=\ell k=\overline{f}r. \]

Finally, since the morphisms $u_X$, $u_Y$, $j_X$, $j_Y$ are anodyne extensions, we get by \cref{lemma:prelim2out3} that $r, s\in \Wf$. Therefore, as $s\hat{f}=\overline{f}r$, by $2$-out-of-$3$ for $\Wf$ we conclude that $\hat{f}\in \Wf$ if and only if $\overline{f}\in \Wf$. 
\end{proof}

We now use the previous lemma to show that the property of a morphism being a weak equivalence is independent of the choice of naive fibrant replacement.

\begin{lem}\label{lemma:wefr}
Let $f \colon X \to Y$ be a morphism in $\W$. Then any naive fibrant replacement 
 $\overline{f} \colon \overline{X} \to \overline{Y}$ of $f$ is in $\Wf$.
\end{lem}

\begin{proof}
Since $f\in\W$, there exists a naive fibrant replacement $f'\colon X'\to Y'$ of $f$ with $f'\in\Wf$. By \cref{lem:univfibrepl}, we have a weakly initial fibrant replacement $\hat{f}$, which is in $\Wf$ as $f'$ is. Hence, \cref{lem:univfibrepl} implies that $\overline{f}$ is also in $\Wf$. 
\end{proof}

If we specialize the above result to the functorial naive fibrant replacement $(-)^\fibrant$ of \cref{rem:naivefibrep}, we get the following. 

\begin{cor}\label{wefromfib}
A morphism $f$ is in $\W$ if and only if $f^\fibrant$ is in $\Wf$.
\end{cor}

We are now equipped to prove condition \ref{Smith:2of3} in Smith's theorem: $\W$ satisfies the 2-out-of-3 property. In fact, we can prove something stronger.\footnote{While this is a stronger property, it is nevertheless expected, as the class of weak equivalences in any model structure will satisfy 2-out-of-6.}

\begin{prop}\label{2of3}
The class $\W$ satisfies the 2-out-of-6 property.
\end{prop}

\begin{proof}
Since $\W$ is the preimage under the functor $(-)^\fibrant$ of the class $\Wf$ by \cref{wefromfib} and $\Wf$ satisfies $2$-out-of-$6$ by condition \ref{2of6Wf}, we directly get that $\W$ satisfies $2$-out-of-$6$. 
\end{proof}

\subsection{Accessibility of \texorpdfstring{$\W$}{W}} 

The accessibility of our class of weak equivalences is a direct consequence of the assumed accessibility of $\Wf$, together with \cref{wefromfib}, as we now describe. This shows condition \ref{Smith:accessibility} of Smith's theorem. 

\begin{prop}\label{Waccessible}
The class $\W$ is accessible as a full subcategory of $\C^{\mathbbm{2}}$. 
\end{prop}
\begin{proof}
Recall that, by assumption, the weak factorization system $(\an,\nfib)$ is generated by a set $\J$. \cref{rem:naivefibrep} describes how this implies the existence of a naive fibrant replacement functor $(-)^\fibrant\colon \C\to \C$ which is accessible.

Now, by condition \ref{accessibility}, we have that the class $\overline{\cW}$ seen as a full subcategory of $\C^{\mathbbm{2}}$ is accessible. Then, by \cref{wefromfib} and since $\Wf$ is the restriction of $\overline{\cW}$ to the morphisms between naive fibrant objects, the subcategory $\W$ is the preimage under the accessible functor $(-)^\fibrant_*\colon \C^{\mathbbm{2}}\to \C^{\mathbbm{2}}$ of the accessible subcategory $\overline{\cW}\subseteq \C^{\mathbbm{2}}$, and so we directly get by \cite[Corollary A.2.6.5]{HTT} that $\W\subseteq \C^{\mathbbm{2}}$ is an accessible subcategory. 
\end{proof}

\subsection{Closure properties of \texorpdfstring{$\cof(\I)\cap\W$}{cof(I) cap W}}

It remains to verify the last condition of Smith's theorem; that is, that $\cof(\I)\cap\W$ is closed under pushout and transfinite composition. We will deduce this from a description of the class $\cof(\I)\cap\W$ in terms of their left lifting property against naive fibrations between naive fibrant objects. It was originally observed by Joyal \cite[Lemma E.2.13]{JoyalVolumeII} that such a description holds in any model structure (where one removes the adjective ``naive'' from our formulation below), and then reformulated by Cisinski \cite[Proposition 1.3.36]{cisinski} in his context of anodyne extensions and naive fibrations. We deduce this description in our setting from a result of Stanculescu \cite[Lemma 1.1(2)]{stanc}. In order to show that the hypotheses of this lemma hold in our context, we highlight an easy result, which is of independent interest.

\begin{lem}\label{lemma:wenfib}
Let $f\colon X\to Y$ be a morphism between naive fibrant objects. Then $f\in\Wf$ if and only if $f\in\W$.
\end{lem}

\begin{proof}
The fact that $\Wf\subseteq\W$ is immediate, considering the naive fibrant replacement given by the identity anodyne extensions. Conversely, suppose that $f\in\W$; then, by definition, there exists a naive fibrant replacement
\begin{tz}
\node[](1) {$X$}; 
\node[right of=1](2) {$Y$}; 
\node[below of=1](1') {$X'$}; 
\node[below of=2](2') {$Y'$}; 

\draw[->] (1) to node[above,la]{$f$} (2); 
\draw[->] (1) to node[left,la]{$\iota_X$} (1'); 
\draw[->] (2) to node[right,la]{$\iota_Y$} (2'); 
\draw[->] (1') to node[below,la]{$f'$} (2');
\end{tz}
with $f'\in\Wf$. But $\iota_X$, $\iota_Y$ are anodyne extensions between naive fibrant objects, and thus belong to $\Wf$ by condition \ref{anwe}. By $2$-out-of-$3$ for $\Wf$, this shows that $f\in\Wf$.
\end{proof}

\begin{prop}\label{trivcoflift}
 A morphism in $\cof(\I)$ is a weak equivalence if and only if it has the left lifting
property with respect to naive fibrations between naive fibrant objects.
\end{prop}
 
\begin{proof}
This is obtained from \cite[Lemma 1.1(2)]{stanc} by taking $(\cA,\cB)=(\cof(\I),\inj(\I))$ and recalling that $\an=\cof(\J)\subseteq \W$ (\cref{ex:anodynewe}), $\W$ has $2$-out-of-$6$ (\cref{2of3}), $\inj(\J)\cap\Wf\subseteq \inj(\I)$ by condition \ref{fibwe}, and lastly that the classes $\W$ and $\Wf$ coincide when we restrict to the morphisms between naive fibrant objects (\cref{lemma:wenfib}).
\end{proof}
           
As a straightforward consequence, we obtain the following result, which is condition \ref{Smith:trivcof} in Smith's theorem.
 
\begin{cor}\label{cor:closedcofwe}
The class $\cof(\I) \cap \W$ is closed under pushout and transfinite composition.
\end{cor}

\subsection{Proof of \cref{thm:main}}

We can finally prove the existence of the proposed model structure.

\begin{proof}[Proof of \cref{thm:main}]
As we see from condition \ref{ax:trivfib},  \cref{2of3,Waccessible,cor:closedcofwe}, the conditions in Smith's theorem are satisfied and thus there exists a combinatorial model structure on $\C$ whose class of cofibrations and weak equivalences are given by $\cof(\I)$ and $\W$, respectively. 

To complete our description, it remains to show that the fibrations between naive fibrant objects consist of the naive fibrations, and hence that the fibrant objects are precisely the naive fibrant objects; and that the weak equivalences between fibrant objects are precisely the morphisms in $\Wf$.

If $f\colon X\to Y$ is a naive fibration between naive fibrant objects, then by \cref{trivcoflift} we have that $f\in (\cof(\I)\cap \W)^{\boxslash}\eqqcolon \fib$. Conversely, since $\an\subseteq\cof(\I)$ by assumption, combining this with \cref{ex:anodynewe} yields $\an\subseteq\cof(\I)\cap\W$; thus, 
\[ \nfib=\an^{\boxslash}\supseteq (\cof(\I)\cap \W)^{\boxslash}= \fib \] as needed. In particular, by applying this to the unique morphisms $f\colon X\to 1$ to the terminal object, we get that fibrant objects are precisely the naive fibrant ones.

The fact that weak equivalences between fibrant objects are precisely the morphisms in $\Wf$ is now the content of \cref{lemma:wenfib}. 
\end{proof}

\subsection{Path object argument}

We conclude this section by proposing an alternative to condition \ref{anwe} in terms of the existence of certain path objects, which has proved to be much simpler to verify in practice.

\begin{prop}\label{pathobj}
Assuming that $\Wf$ has $2$-out-of-$6$, the following two conditions are equivalent:
\begin{enumerate}[label=(\Alph*)]
\setcounter{enumi}{15}
\item[\ref{anwe}] the morphisms in $\an$ between naive fibrant objects are in $\Wf$,
    \item \label{Path} for every naive fibrant object $X$, there is a factorization of the diagonal morphism 
    \[ X\xrightarrow{w} \Path X\xrightarrow{p} X\times X \]
    such that $w\in \Wf$ and $p\in \nfib$.
\end{enumerate}
\end{prop}

\begin{rem} \label{rem:piwe}
We write $p_i\colon \Path X\to X$ for the composite 
\[ \Path X\xrightarrow{p} X\times X\xrightarrow{\pi_i} X, \] 
where $\pi_i$ denotes the projection for $i=0,1$. Since $w\in \Wf$ and $p_i  w=\id_X$, we get that $p_i\in \Wf$ by $2$-out-of-$3$.
\end{rem}

\begin{proof}
Suppose that \ref{anwe} is satisfied. We factor the diagonal morphism at $X$ as 
\[ X\xrightarrow{w} \Path X\xrightarrow{p} X\times X \] 
with $w\in \an$ and $p\in \nfib$. Note that $w$ is then an anodyne extension between naive fibrant objects, and thus by \ref{anwe} we get that $w\in\Wf$, showing \ref{Path}. 

Now suppose that \ref{Path} holds, and let $f\colon X\to Y$ be an anodyne extension with $X,Y$ naive fibrant. Then there is a lift in the diagram below left. 
\begin{tz}
\node[](1) {$X$}; 
\node[below of=1](2) {$Y$}; 
\node[right of=1](3) {$X$}; 
\draw[->] (1) to node[above,la]{$\id_X$} (3);
\draw[->] (1) to node[left,la]{$f$} (2); 
\draw[->,dashed] (2) to node[below,la,xshift=2pt]{$s$} (3);

\node[right of=1,xshift=3cm](1) {$X$}; 
\node[below of=1](2) {$Y$}; 
\node[right of=1](3) {$Y$}; 
\node[right of=3](3') {$\Path Y$}; 
\node[below of=3'](4) {$Y\times Y$};
\draw[->] (1) to node[above,la]{$f$} (3);
\draw[->] (3) to node[above,la]{$w$} (3');
\draw[->] (1) to node[left,la]{$f$} (2); 
\draw[->] (3') to node[right,la]{$p$} (4);
\draw[->,dashed] (2) to node[below,la]{$g$} (3');
\draw[->] (2) to node[below,la]{$(\id_Y,f s)$} (4);
\end{tz}
Since $f\in \an$ and $p\colon \Path Y\to Y\times Y$ is in $\nfib$, there is also a lift in the commutative diagram above right. 

Now, the morphism $p_0\colon \Path Y\to Y$ is in $\Wf$ by \cref{rem:piwe}, and so the equality $p_0 g=\id_Y$ gives $g\in \Wf$ by $2$-out-of-$3$. Using the fact that $p_1\colon \Path Y\to Y$ is in $\Wf$ and that $p_1 g=f s$, we further get that $f s\in\Wf$ by $2$-out-of-$3$. Finally, applying $2$-out-of-$6$ to the following diagram 
\begin{tz}
\node[](1) {$X$}; 
\node[below of=1](2) {$Y$}; 
\node[right of=1](3) {$X$};  
\node[below of=3](4) {$Y$};
\draw[->] (1) to node[above,la]{$\id_X$} (3);
\draw[->] (1) to node[left,la]{$f$} (2); 
\draw[->] (3) to node[right,la]{$f$} (4);
\draw[->] (2) to node[below,la,xshift=2pt]{$s$} (3);
\draw[->] (2) to node[below,la]{$f s$} (4);
\end{tz}
we conclude that $f\in \Wf$ as desired. 
\end{proof}

\section{Fibrantly-transferred model structures along an adjunction}

The main source of applications of \cref{thm:main} that we envision, and in fact what led us to consider this problem initially, is the case when the cofibrations, and the fibrant objects and fibrations and weak equivalences between them, are transferred through a right adjoint functor to a known model structure. Concretely, we consider the following scenario.

Let $(\cM,\cofM,\fibM,\weM)$ be a combinatorial model category with generating sets of cofibrations and trivial cofibrations denoted by $\IM$ and $\JM$, and let $\C$ be a locally presentable category. Suppose that we have an adjunction 
\begin{tz}
\node[](1) {$\C$}; 
\node[right of=1,xshift=1cm](2) {$\cM$.}; 

\draw[->] ($(1.east)-(0,5pt)$) to node[below,la]{$R$} ($(2.west)-(0,5pt)$);
\draw[->] ($(2.west)+(0,5pt)$) to node[above,la]{$L$} ($(1.east)+(0,5pt)$);

\node[la] at ($(1.east)!0.5!(2.west)$) {$\bot$};
\end{tz}

We can use this adjunction to define the classes of morphisms involved in the setup of \cref{thm:main}.

\begin{defn}
We define $\I=L\IM$ to be the generating set of cofibrations in $\C$. 
\end{defn}

Note that $\inj(\I)=R^{-1}(\trivfibM)$ and so we have a weak factorization system \[ (\cof(\I),R^{-1}(\trivfibM)) \] in $\C$ induced from the weak factorization system $(\cof(\IM), \trivfibM)$ in $\cM$. 

\begin{defn}
We define $\J=L\JM$ to be the generating set of anodyne extensions in~$\C$. 
\end{defn}

Note that $\inj(\J)=R^{-1}(\fibM)$ and so we have a weak factorization system \[ \cof(L\JM), R^{-1}(\fibM)) \]
in $\C$ induced from the weak factorization system $(\cof(\JM),\fibM)$ in $\cM$; this will be the factorization system that we denote by $(\an,\nfib)$.

\begin{rem}
An object $X\in\C$ is naive fibrant if and only if $RX$ is fibrant in $\cM$. 
\end{rem}

\begin{defn}
A morphism $f$ in $\C$ between naive fibrant objects is a weak equivalence if $f\in R^{-1}(\weM)$. We denote this class of morphisms by $\Wf$.
\end{defn}

Just as we did in \cref{def:we}, we can construct a class of weak equivalences $\W$ by using the class $\Wf$ and the naive fibrant replacements given by the weak factorization system $(\an,\nfib)$. With these classes of morphisms, we can now recast our main result under the light of an adjunction. We refer to the resulting model structure as a \emph{fibrantly-transferred model structure}.

\begin{theorem}\label{thm:semirightinduced}
Let $(\cM,\cofM,\fibM,\weM)$ be a combinatorial model category, and let $\C$ be a locally presentable category. Suppose that we have an adjunction 
\begin{tz}
\node[](1) {$\C$}; 
\node[right of=1,xshift=1cm](2) {$\cM$}; 

\draw[->] ($(1.east)-(0,5pt)$) to node[below,la]{$R$} ($(2.west)-(0,5pt)$);
\draw[->] ($(2.west)+(0,5pt)$) to node[above,la]{$L$} ($(1.east)+(0,5pt)$);

\node[la] at ($(1.east)!0.5!(2.west)$) {$\bot$};
\end{tz}
and that the following properties are satisfied:
\begin{enumerate}[label=(\arabic*)]
    \item\label{srind2} $R^{-1}(\trivfibM)\subseteq \W$,
    \item\label{srind1} for every naive fibrant object $X$, there is a factorization of the diagonal morphism 
    \[ X\xrightarrow{w} \Path X\xrightarrow{p} X\times X \]
    such that $w\in R^{-1}(\weM)$ and $p\in R^{-1}(\fibM)$.
\end{enumerate}

Then, there exists a combinatorial model structure on $\C$ in which a morphism $f$ is a trivial fibration if and only if $Rf$ is so in $\cM$, and an object $X$ is fibrant if and only if $RX$ is fibrant in $\cM$. Moreover, a morphism $f$ between fibrant objects in $\C$ is a weak equivalence (resp.\ fibration) if and only if $Rf$ is so in $\cM$.
\end{theorem}

\begin{proof}
It suffices to apply \cref{thm:main} to the classes defined above. First, the fact that $\JM\subseteq \cof(\IM)$ implies that $L\JM\subseteq L\cof(\IM)$ and, moreover, as $L$ commutes with colimits, we have $L\cof(\IM)\subseteq \cof(L\IM)$. Hence, as $\cof(L\IM)$ is saturated, we have
\[\an\coloneqq \cof(L\JM)\subseteq\cof(L\IM)=\cof(\I). \]
Next note that condition \ref{ax:trivfib} is precisely the first hypothesis in this theorem, and condition \ref{anwe} agrees with the second hypothesis by \cref{pathobj}. The fact that $\Wf$ satisfies $2$-out-of-$6$ is straightforward from its definition as the restriction of $R^{-1}(\weM)$ to naive fibrant objects, since the class of weak equivalences $\weM$ in any model category satisfies $2$-out-of-$6$. Choosing $\overline{\cW}\coloneqq R^{-1}(\weM)$ gives an accessible full subcategory of $\C^{\mathbbm{2}}$. Indeed, as $\cM$ is combinatorial and $\C$ is locally presentable, then $\weM$ is an accessible full subcategory of $\cM^{\mathbbm{2}}$ and as $R$ is an accessible functor we can conclude by \cite[Corollary A.2.6.5]{HTT}. Finally 
\[ \nfib\cap\Wf\subseteq R^{-1}(\fibM)\cap R^{-1}(\weM)=R^{-1}(\fibM\cap\weM)=\inj(L\IM). \qedhere \]
\end{proof}

\begin{rem}
We chose to formulate \cref{thm:semirightinduced} using a path object condition since we find this  easier to verify in practice, but of course, by \cref{pathobj} we could replace condition \ref{srind1} above by the following requirement:
\begin{enumerate}
    \item[(2')] every morphism $f\in \cof(L\J)$ between naive fibrant objects is in $R^{-1}(\weM)$.
\end{enumerate} 
\end{rem}

As expected, the adjunction through which we transfer the model structure becomes a Quillen adjunction.

\begin{prop}
In the setting of \cref{thm:semirightinduced}, the adjunction
\begin{tz}
\node[](1) {$\C$}; 
\node[right of=1,xshift=1cm](2) {$\cM$.}; 

\draw[->] ($(1.east)-(0,5pt)$) to node[below,la]{$R$} ($(2.west)-(0,5pt)$);
\draw[->] ($(2.west)+(0,5pt)$) to node[above,la]{$L$} ($(1.east)+(0,5pt)$);

\node[la] at ($(1.east)!0.5!(2.west)$) {$\bot$};
\end{tz}
is a Quillen pair.
\end{prop}
\begin{proof}
The functor $R$ preserves trivial fibrations by definition. To see that $R$ preserves fibrations, recall that anodyne extensions are included in $\cof(\I)$ by construction, and in $\W$ by \cref{ex:anodynewe}; then, \[ \fib=(\cof(\I)\cap\W)^\boxslash \subseteq\an^\boxslash=R^{-1}(\fibM). \qedhere \]
\end{proof}

\begin{rem}\label{rmk:comparison}
The classical approach in the presence of an adjunction is to right-transfer the model structure along the right adjoint $R\colon \cC\to \cM$, where the class of fibrations and weak equivalences in $\cC$ are given by $R^{-1}(\fibM)$ and $R^{-1}(\weM)$, respectively. However, as we will see in \cref{application:cat}, there are settings where the right-transferred model structure does not exist, while the fibrantly-transferred one does. If the right-transferred model structure exists, it coincides with the fibrantly-transferred one, as they have the same cofibrations and fibrant objects (see \cite[Proposition E.1.10]{JoyalVolumeII}).

A dual version of \cite[Theorem 2.2.1]{HKRS} inspired by Quillen's original path object argument \cite{Quillen} (see also \cite[Theorem 6.2]{Moser} for a precise statement) gives a useful criterion for the existence of a right-transferred model structure. When comparing our result to this statement, we can see that we no longer require that fibrant replacements are weak equivalences, which is automatic in our framework, but instead that trivial fibrations are weak equivalences, which is automatic in the right-transferred framework. 
\end{rem}

\begin{rem}
    Note that the set $\cJ_\cM$ is only needed to define the fibrations between fibrant objects in $\C$. Hence, throughout this section, one could replace the set $\cJ_\cM$ of generating trivial cofibrations of $\cM$ with a set of generating anodyne extensions of $\cM$, i.e., a set of morphisms which determine the fibrations between fibrant objects of $\cM$. This is very useful in practice, as we often only know fibrations between fibrant objects.
\end{rem}

\section{A model structure for double categories along the horizontal nerve} \label{application:cat}

In \cite{FPP}, Fiore\textendash Paoli\textendash Pronk construct several model structures on $\dblcat$, all closely related to the canonical model structure on $\cat$. As part of this program, they consider the horizontal nerve $N^h\colon \dblcat\to \cat^{\Dop}$ to relate them. Since $N^h$ is a right adjoint, and it is natural to consider the Reedy model structure on $\cat^{\Dop}$, a question arises of whether we can use this functor to right-transfer a model structure on $\dblcat$. However, they find that this is not possible.

After a brief recollection of the relevant double categorical preliminaries, in this section we show how the added flexibility of \cref{thm:semirightinduced} allows us to fibrantly-transfer a model structure on $\dblcat$ through the horizontal nerve $N^h$ from the Reedy model structure on~$\cat^{\Dop}$. Note that this does not change the proposed homotopy theory, as the fibrant objects and weak equivalences between them determine the homotopy category of a model structure.

\subsection{Double categorical background}\label{section:background}

A reader who is familiar with the theory of double categories is welcome to skip this section, as it does not introduce any new notions. On the other hand, a reader looking for a more thorough introduction to double categories may find many more details and examples in \cite[Chapter 3]{grandis}.

\begin{defn}
A \emph{double category} $\bA$ consists of objects $A,B,A',B',\ldots$, horizontal morphisms $f\colon A\to B$, vertical morphisms $u\colon A\arrowdot A'$, and squares
\begin{tz}
\node[](1) {$A$}; 
\node[below of=1](2) {$A'$}; 
\node[right of=1](3) {$B$}; 
\node[right of=2](4) {$B'$};

\draw[->,pro] (1) to node[left,la]{$u$} (2); 
\draw[->] (1) to node[above,la]{$f$} (3); 
\draw[->] (2) to node[below,la]{$f'$} (4); 
\draw[->,pro](3) to node[right,la]{$v$} (4); 
 
\node[la] at ($(1)!0.5!(4)$) {$\alpha$};
\end{tz}
together with associative and unital compositions for horizontal morphisms, vertical morphisms, and squares. We denote by $\id_A$ (resp.\ $e_A$) the horizontal (resp.\ vertical) identity at an object $A$,  by $\id_u$ (resp.\ $e_f$) the horizontal (resp.\ vertical) identity square at a vertical morphism $u$ (resp.\ horizontal morphism $f$), and we write $\square_A=\id_{e_A}=e_{\id_A}$ for the identity square at an object $A$. 

A \emph{double functor} $F\colon \cA\to \cB$ is an assignment on objects, horizontal morphisms, vertical morphisms, and squares that preserve all compositions and identities strictly. 

We denote by $\dblcat$ the category of double categories and double functors.
\end{defn}

There are several ways of seeing a %$2$-
category as a double category.  An intuitive way to do this is to encode the information of the %$2$-
category in the horizontal direction, while allowing only trivial vertical morphisms and squares.

\begin{defn}
The \emph{horizontal embedding} functor $\bH\colon \cat\to \dblcat$ sends a category $\cA$ to the double category $\bH\cA$ whose objects and horizontal morphisms are the objects and morphisms of $\cA$, while the vertical morphisms and squares are trivial. 
\end{defn}

\begin{rem}
The functor defined above admits a right adjoint $\bfH\colon \dblcat\to \cat$ which extracts from a double category $\bA$ its \emph{underlying horizontal category} $\bfH\bA$ forgetting the vertical morphisms and squares of $\bA$.
\end{rem}

\begin{rem}
By interchanging the horizontal and vertical directions, we similarly have adjoints $\bV\colon \dblcat\to \cat$ and $\bfV\colon \dblcat\to \cat$.
\end{rem}

The category of double categories is cartesian closed, and the corresponding internal hom will play a role in the model structure we construct in this section. We now recall the relevant definitions.

\begin{defn}
A \emph{horizontal natural transformation} $\kappa\colon F\Rightarrow G$ between double functors $F,G\colon \bA\to \bB$ consists of
\begin{rome}
    \item a horizontal morphism $\kappa_A\colon FA\to GA$ in $\bB$, for each object $A\in \bA$, 
    \item a square $\kappa_u$ in $\bB$ as below, for each vertical morphism $u\colon A\arrowdot A'$ in $\bA$, 
    \begin{tz}
    \node[](1) {$FA$}; 
\node[below of=1](2) {$FA'$}; 
\node[right of=1](3) {$GA$}; 
\node[below of=3](4) {$GA'$}; 

\draw[->,pro] (1) to node[left,la]{$Fu$} (2);
\draw[->] (2) to  node[below,la]{$\kappa_{A'}$} (4); 
\draw[->] (1) to node[above,la]{$\kappa_{A}$} (3); 
\draw[->,pro] (3) to node[right,la]{$Gu$} (4); 

\node[la] at ($(1)!0.5!(4)$) {$\kappa_u$};
\end{tz}
\end{rome} 
such that the squares in (ii) are compatible with vertical compositions and identities, and these data satisfy a naturality condition with respect to horizontal morphisms
and squares.%and the squares in (iii) are compatible with horizontal compositions and identities. Together, they satisfy a pseudo naturality condition with respect to squares in $\bA$. We say that $\kappa$ is a \emph{horizontal (strict) natural transformation} if the $2$-isomorphisms $\kappa_f$ are identities for all horizontal morphisms $f$ in $\bA$. 

By reversing the horizontal and vertical directions, one can define \emph{vertical natural transformations}. 

A \emph{modification} $\mu$ between two horizontal  natural transformations $\kappa,\kappa'$ and two vertical natural transformations $\lambda,\lambda'$  as below left consists of a square $\mu_A$ in $\bB$ as below right, for each object $A\in \bA$, 
\begin{tz}
    \node[](1) {$F$}; 
\node[below of=1](2) {$F'$}; 
\node[right of=1](3) {$G$}; 
\node[below of=3](4) {$G'$}; 

\draw[pro,n] (1) to node[left,la]{$\lambda$} (2);
\draw[n] (2) to  node[below,la]{$\kappa'$} (4); 
\draw[n] (1) to node[above,la]{$\kappa$} (3); 
\draw[pro,n] (3) to node[right,la]{$\lambda'$} (4); 

\node[la] at ($(1)!0.5!(4)$) {$\mu$};

\node[right of=3,xshift=2cm](1) {$FA$}; 
\node[below of=1](2) {$F'A$}; 
\node[right of=1](3) {$GA$}; 
\node[below of=3](4) {$G'A$}; 

\draw[pro,->] (1) to node[left,la]{$\lambda_A$} (2);
\draw[->] (2) to  node[below,la]{$\kappa'_A$} (4); 
\draw[->] (1) to node[above,la]{$\kappa_A$} (3); 
\draw[pro,->] (3) to node[right,la]{$\lambda'_A$} (4); 

\node[la] at ($(1)!0.5!(4)$) {$\mu_A$};
\end{tz}
satisfying horizontal and vertical coherence conditions with respect to the square components of the natural transformations.
\end{defn}

The above data assemble to form double categories of double functors. 

\begin{defn} \label{homs}
    Let $\bA,\bB$ be double categories. The \emph{hom double category} $\llbracket \bA,\bB\rrbracket$ is the double category of double functors $\bA\to \bB$, horizontal and vertical natural transformations, and modifications. %The \emph{pseudo hom double category} $\llbracket \bA,\bB\rrbracket_\mathrm{ps}$ is the double category of double functors $\bA\to \bB$, horizontal and vertical pseudo natural transformations, and modifications.
\end{defn}

\subsection{Right- vs.\ fibrantly-transferred model structure}

The following also appears as \cite[Definition 5.1]{FPP} using the description of \cite[Proposition 5.3]{FPP}.

\begin{defn}\label{defn:hornerve}
The \emph{horizontal nerve} $N^h\colon \dblcat\to \cat^{\Dop}$ is the functor sending a double category $\bA$ to the simplicial object in $\cat$
\[ N^h\bA\colon \Dop\to \cat, \quad [n]\mapsto \bfV\llbracket\bH [n], \bA\rrbracket. \]
More explicitly, $(N^h\bA)_0=\bA_0$ is the category of objects and vertical morphisms of $\bA$, $(N^h\bA)_1=\bA_1$ is the category of horizontal morphisms and squares of $\bA$, and for $n\geq 2$
\[ (N^h\bA)_n=\bA_1\times_{\bA_0}\ldots \times_{\bA_0} \bA_1 \]
is the category of $n$ composable horizontal morphisms and $n$ horizontally composable squares in $\bA$. 
\end{defn}

\begin{prop}\cite[Theorem 5.6]{FPP}
The functor $N^h\colon \dblcat\to \cat^{\Dop}$ admits a left adjoint $c^h\colon \cat^{\Dop}\to \dblcat$.
\end{prop}

We endow the category $\cat^{\Dop}$ with the Reedy model structure \cite[\S 15.3]{Hirsch} on simplicial objects in the canonical model structure on $\cat$. Recall that the weak equivalences in this model structure are the level-wise weak equivalences; i.e., the morphisms $\cX\to \cY$ in $\cat^{\Dop}$ such that $\cX_n\to \cY_n$ is an equivalence of categories, for all $n\geq 0$. 

As established in \cite{FPP}, it is not possible to right-transfer a model structure on $\dblcat$ through the nerve functor $N^h$.

\begin{prop}\cite[Theorem 7.22]{FPP}
Suppose that $\cat^{\Dop}$ is endowed with the Reedy model structure on simplicial objects in the canonical model structure on $\cat$. Then the right-transferred model structure on $\dblcat$ along the horizontal nerve functor \[N^h \colon \dblcat\to \cat^{\Dop} \] 
does not exist.
\end{prop}

The failure of this model structure to exist can also be understood through an example. 

\begin{ex}
If the right-transferred model structure were to exist, the pushout-product $\bH\mathbbm{2}\sqcup_{\mathbbm{1}\sqcup\mathbbm{1}} (\bV \bI\sqcup \bV\bI)\to \bV\bI\times \bH\mathbbm{2}$ would be a trivial cofibration (see \cref{tab:genanext}), and therefore every pushout of this map would be a weak equivalence. However, this is not the case.

Given the double category $\bA$ generated by the data below left, and the pushout as below right, we show that the morphism $j\colon \bA\to\mathbb{P}$ is not sent to a weak equivalence in $\cat^{\Dop}$ by $N^h$.
\begin{tz}
\node[](1) {$A$}; 
\node[right of=1](2) {$B$}; 
\node[below of=2](2') {$B'$}; 
\node[right of=2'](3') {$C'$}; 

\draw[->] (1) to node[above,la]{$f$} (2); 
\draw[->,pro] (2) to node[left,la]{$u$} node[above,la,sloped]{$\cong$} (2'); 
\draw[->] (2') to node[below,la]{$g'$} (3');

\node[right of=2,xshift=3cm](1) {$\bH\mathbbm{2}\sqcup_{\mathbbm{1}\sqcup\mathbbm{1}} (\bV\bI\sqcup\bV\bI)$};
\node[below of=1](2) {$\bV\bI\times \bH\mathbbm{2}$}; 
\node[right of=1,xshift=2.5cm](3) {$\bA$}; 
\node[below of=3](4) {$\bP$}; 
\node at ($(4)-(.3cm,-.3cm)$) {$\ulcorner$};

\draw[->] (1) to (2); 
\draw[->] (1) to node[above,la]{$g'+(u+e_{C'})$} (3); 
\draw[->] (2) to node[below,la]{$\varphi$} (4);
\draw[->] (3) to node[right,la]{$j$} (4);
\end{tz}
The pushout $\bP$ is the double category generated by the following data.
\begin{tz}
\node[](0) {$A$}; 
\node[right of=0](1) {$B$}; 
\node[below of=1](2) {$B'$}; 
\node[right of=1](3) {$C'$}; 
\node[right of=2](4) {$C'$};

\draw[->,pro] (1) to node[left,la]{$u$} node[above,la,sloped]{$\cong$} (2); 
\draw[->] (0) to node[above,la]{$f$} (1); 
\draw[->] (1) to node[above,la]{$g$} (3); 
\draw[->] (2) to node[below,la]{$g'$} (4); 
\draw[d,pro](3) to (4); 

\node[la] at ($(1)!0.5!(4)+(7pt,0)$) {\rotatebox{270}{$\cong$}}; 
\node[la] at ($(1)!0.5!(4)-(3pt,0)$) {$\varphi$};
\end{tz}
Using this description, one can see that the functor $(N^h\bA)_1\to (N^h\bP)_1$ is not essentially surjective on objects since there are no horizontal morphisms in $\bA$ that are related by a vertically invertible square to the composite $gf$ in $\bP$. 
\end{ex}

This shows that, if there exists a model structure on $\dblcat$ whose trivial fibrations are right-transferred, it necessarily has more weak equivalences than the right-transferred ones. More specifically, we claim that the right-transferred weak equivalences are only the correct class when working between fibrant objects, and one must then enlarge this class as in \cref{def:we} to be able to consider all objects. Indeed, using \cref{charnaivefibNh} it is not hard to check that the double categories $\bA$ and $\bP$ considered above are not fibrant. In the remainder of this section, we show how we can apply \cref{thm:semirightinduced} to obtain the required model structure.

\begin{theorem} \label{thm:SRIalongnerve}
Suppose that $\cat^{\Dop}$ is endowed with the Reedy model structure on simplicial objects in the canonical model structure on $\cat$. Then the fibrantly-transferred model structure on $\dblcat$ along the horizontal nerve $N^h\colon \dblcat\to \cat^{\Dop}$ exists. 
\end{theorem}

Note that this is not a significant change for the purposes of the homotopy theory we intended to model through the right adjoint $N^h \colon \dblcat\to \cat^{\Dop}$, as the fibrant objects and weak equivalences between them determine the homotopy category of a model structure. Hence, although some of the classes of maps must change in order for the desired model structure to exist, the proposed homotopy theory remains the same.

\subsection{Cofibrations and anodyne extensions}

We first wish to compute the generating cofibrations and generating anodyne extensions of the fibrantly-transferred model structure. These sets are given by the image of sets of generating cofibrations and trivial cofibrations in $\cat^{\Dop}$ under the left adjoint $c^h\colon \cat^{\Dop}\to \dblcat$ of the horizontal nerve. 

\begin{rem}
Recall that a set of generating cofibrations for the canonical model structure on $\cat$ consists of the following functors:
\begin{rome}
    \item the unique functor $\emptyset\to \mathbbm{1}$, 
    \item the inclusion $\mathbbm{1}\sqcup\mathbbm{1}\to \mathbbm{2}$, 
    \item the functor $\mathbbm{2}\sqcup_{\mathbbm{1}\sqcup\mathbbm{1}} \mathbbm{2}=\{ \boldsymbol{\cdot}\rightrightarrows\boldsymbol{\cdot}\}\to \mathbbm{2}$ sending the two parallel morphisms to the non-trivial morphism in $\mathbbm{2}$. 
\end{rome}
A set of generating trivial cofibrations for the canonical model structure on $\cat$ consists of the single inclusion functor
 \[ \mathbbm{1}\to \bI=\{ \boldsymbol{\cdot}\cong \boldsymbol{\cdot}\} \]
sending the object to one of the end-points of the free-living isomorphism. 

Then, by \cite[Theorem 15.6.27]{Hirsch}, a set of generating (trivial) cofibrations for the Reedy model structure on $\cat^{\Dop}$ is given by the pushout-product morphisms
\[ \cC\times \Delta[n]\sqcup_{\cC\times \partial\Delta[n]} \cD\times \partial\Delta[n]\to \cD\times \Delta[n] \]
with $\cC\to \cD$ a generating (trivial) cofibration in $\cat$ and $n\geq 0$.
\end{rem}

Write $c\colon \Set^{\Dop}\to \cat$ for the left adjoint of the standard nerve $N\colon \cat\to \Set^{\Dop}$.

\begin{lem}\cite[Proposition 6.11]{FPP} \label{lem:computec}
If $\cC\in \cat\subseteq \cat^{\Dop}$ and $X\in \Set^{\Dop}\subseteq \cat^{\Dop}$, then there is a natural isomorphism in $\dblcat$
\[ c^h(\cC\times X)\cong \bV\cC\times \bH c X. \]
\end{lem}

\begin{lem}
    For $\cC\to \cD$ a functor in $\cat$ and $n\geq 0$, the functor $c^h\colon \cat^{\Dop}\to \dblcat$ sends the pushout-product 
\[ \cC\times \Delta[n]\sqcup_{\cC\times \partial\Delta[n]} \cD\times \partial\Delta[n]\to \cD\times \Delta[n] \]
of $\cC\to \cD$ with $\partial\Delta[n]\to \Delta[n]$ to the pushout-product 
\[ \bV\cC\times \bH c (\Delta[n])\sqcup_{\bV\cC\times \bH c (\partial\Delta[n])} \bV\cD\times \bH c(\partial\Delta[n])\to \bV\cD\times \bH c(\Delta[n]) \]
of $\bV(\cC\to \cD)$ with $\bH c(\partial\Delta[n]\to \Delta[n])$.
\end{lem}

\begin{proof}
    This follows from the fact that $c^h$ commutes with colimits, and \cref{lem:computec}.
\end{proof}

We can now compute the generating cofibrations and generating anodyne extensions of the model structure on $\dblcat$. Since $c(\partial\Delta[n]\to \Delta[n])$ is the identity for $n\geq 3$ we only need to compute the pushout-products when $n\leq 2$. In these cases, we have that the functor $c(\partial\Delta[n]\to \Delta[n])$ 
\begin{itemize}
    \item when $n=0$, is given by the unique functor $\emptyset\to \mathbbm{1}$, 
    \item when $n=1$, is given by the inclusion functor $\mathbbm{1}\sqcup\mathbbm{1}\to \mathbbm{2}$, 
    \item when $n=2$, can be replaced by the functor $\mathbbm{2}\sqcup_{\mathbbm{1}\sqcup\mathbbm{1}} \mathbbm{2}=\{\boldsymbol{\cdot}\rightrightarrows\boldsymbol{\cdot}\}\to \mathbbm{2}$. Indeed, we have that $\{\boldsymbol{\cdot}\rightrightarrows\boldsymbol{\cdot}\}\to \mathbbm{2}$ is a retract of $c(\partial\Delta[2]\to \Delta[2])$, and that $c(\partial\Delta[2]\to \Delta[2])$ is a pushout of $\{\boldsymbol{\cdot}\rightrightarrows\boldsymbol{\cdot}\}\to \mathbbm{2}$.
\end{itemize}

Hence the generating cofibrations and generating anodyne extensions can be computed as in \cref{tab:gencof,tab:genanext}, where we depict the pushout-products of the morphisms in the first rows and columns.

\begin{table}[ht]
    \centering
    \begin{tabularx}{0.98\textwidth}{c|c|c|c}
    $\Set^{\Dop} \backslash \cat$ & $\emptyset\to \mathbbm{1}$ & $\mathbbm{1}\sqcup\mathbbm{1}\to \bV\mathbbm{2}$ & $\bV(\boldsymbol{\cdot} \rightrightarrows \boldsymbol{\cdot})\to \bV\mathbbm{2}$ \\ \hline 
   $\emptyset\to \mathbbm{1}$ & \parbox[c]{1.35cm}{\begin{tikzpicture}[node distance=1cm]
    \node[](1) {}; 
    \node[right of=1](2) {$\boldsymbol{\cdot}$}; 
    \node at ($(1)!0.5!(2)$) {$\to$};
    \end{tikzpicture}} & \parbox[c]{1.35cm}{\begin{tikzpicture}[node distance=1cm]
    \node[](1) {$\boldsymbol{\cdot}$};
    \node[below of=1](2) {$\boldsymbol{\cdot}$}; 
    \node[right of=1](1) {$\boldsymbol{\cdot}$};
    \node at ($(1)!0.5!(2)$) {$\to$};
    \node[below of=1](2) {$\boldsymbol{\cdot}$}; 
    \draw[->,prosmall] (1) to (2);
    \end{tikzpicture}}& \parbox[c]{1.7cm}{\begin{tikzpicture}[node distance=1cm]
    \node[](1) {$\boldsymbol{\cdot}$};
    \node[below of=1](2) {$\boldsymbol{\cdot}$}; 
    \draw[->,prosmall] ($(1.south)-(3pt,0)$) to ($(2.north)-(3pt,0)$);
    \draw[->,prosmall] ($(1.south)+(3pt,0)$) to ($(2.north)+(3pt,0)$);
    \node[right of=1,xshift=.3cm](1) {$\boldsymbol{\cdot}$};
    \node at ($(1)!0.5!(2)$) {$\to$};
    \node[below of=1](2) {$\boldsymbol{\cdot}$}; 
    \draw[->,prosmall] (1) to (2);
    \end{tikzpicture}} \\ \hline
    $\mathbbm{1}\sqcup\mathbbm{1}\to \bH\mathbbm{2}$ & \parbox[c]{3.4cm}{\begin{tikzpicture}[node distance=1cm]
    \node[](1) {$\boldsymbol{\cdot}$};
    \node[right of=1](2) {$\boldsymbol{\cdot}$}; 
    \node[right of=2](1) {$\boldsymbol{\cdot}$};
    \node at ($(1)!0.5!(2)$) {$\to$};
    \node[right of=1](2) {$\boldsymbol{\cdot}$}; 
    \draw[->] (1) to (2);
    \end{tikzpicture}} & \parbox[c]{3.4cm}{\begin{tikzpicture}[node distance=1cm]
    \node[](1) {$\boldsymbol{\cdot}$}; 
    \node[right of=1](2) {$\boldsymbol{\cdot}$};
    \node[below of=1](1') {$\boldsymbol{\cdot}$};
    \node[below of=2](2') {$\boldsymbol{\cdot}$};
    \draw[->] (1) to (2);
    \draw[->] (1') to (2');
    \draw[->,prosmall] (1) to (1');
    \draw[->,prosmall] (2) to (2');
    \node[right of=2](1) {$\boldsymbol{\cdot}$}; 
    \node at ($(1)!0.5!(2')$) {$\to$}; 
    \node[right of=1](2) {$\boldsymbol{\cdot}$};
    \node[below of=1](1') {$\boldsymbol{\cdot}$};
    \node[below of=2](2') {$\boldsymbol{\cdot}$};
    \draw[->] (1) to (2);
    \draw[->] (1') to (2');
    \draw[->,prosmall] (1) to (1');
    \draw[->,prosmall] (2) to (2');
    \node[la] at ($(1)!0.5!(2')$) {\rotatebox{270}{$\Rightarrow$}};
    \end{tikzpicture}} & \parbox[c]{3.7cm}{\begin{tikzpicture}[node distance=1cm]
    \node[](1) {$\boldsymbol{\cdot}$}; 
    \node[right of=1](2) {$\boldsymbol{\cdot}$};
    \node[below of=1](1') {$\boldsymbol{\cdot}$};
    \node[below of=2](2') {$\boldsymbol{\cdot}$};
    \draw[->] (1) to (2);
    \draw[->] (1') to (2');
    \draw[->,prosmall] ($(1.south)-(3pt,0)$) to ($(1'.north)-(3pt,0)$);
    \draw[->,prosmall] ($(1.south)+(3pt,0)$) to ($(1'.north)+(3pt,0)$);
    \draw[->,prosmall] ($(2.south)-(3pt,0)$) to ($(2'.north)-(3pt,0)$);
    \draw[->,prosmall] ($(2.south)+(3pt,0)$) to ($(2'.north)+(3pt,0)$);
    \node[la] at ($(1)!0.5!(2')-(3.5pt,0)$) {\rotatebox{270}{$\Rightarrow$}};
    \node[la] at ($(1)!0.5!(2')+(3.5pt,0)$) {\rotatebox{270}{$\Rightarrow$}};
    \node[right of=2,xshift=.3cm](1) {$\boldsymbol{\cdot}$}; 
    \node at ($(1)!0.5!(2')$) {$\to$}; 
    \node[right of=1](2) {$\boldsymbol{\cdot}$};
    \node[below of=1](1') {$\boldsymbol{\cdot}$};
    \node[below of=2](2') {$\boldsymbol{\cdot}$};
    \draw[->] (1) to (2);
    \draw[->] (1') to (2');
    \draw[->,prosmall] (1) to (1');
    \draw[->,prosmall] (2) to (2');
    \node[la] at ($(1)!0.5!(2')$) {\rotatebox{270}{$\Rightarrow$}};
    \end{tikzpicture}}  \\ \hline
    $\bH(\boldsymbol{\cdot} \rightrightarrows \boldsymbol{\cdot})\to \bH\mathbbm{2}$ & \parbox[c]{3.4cm}{\begin{tikzpicture}[node distance=1cm]
    \node[](1) {$\boldsymbol{\cdot}$};
    \node[right of=1](2) {$\boldsymbol{\cdot}$}; 
    \draw[->] ($(1.east)-(0,3pt)$) to ($(2.west)-(0,3pt)$);
    \draw[->] ($(1.east)+(0,3pt)$) to ($(2.west)+(0,3pt)$);
    \node[right of=2](1) {$\boldsymbol{\cdot}$};
    \node at ($(1)!0.5!(2)$) {$\to$};
    \node[right of=1](2) {$\boldsymbol{\cdot}$}; 
    \draw[->] (1) to (2);
    \end{tikzpicture}} & \parbox[c]{3.4cm}{\begin{tikzpicture}[node distance=1cm]
    \node[](1) {$\boldsymbol{\cdot}$}; 
    \node[above of=1, yshift=-.8cm] {};
    \node[right of=1](2) {$\boldsymbol{\cdot}$};
    \node[below of=1](1') {$\boldsymbol{\cdot}$};
    \node[below of=2](2') {$\boldsymbol{\cdot}$};
    \node[below of=1', yshift=.8cm] {};
    \draw[->] ($(1.east)-(0,3pt)$) to ($(2.west)-(0,3pt)$);
    \draw[->] ($(1.east)+(0,3pt)$) to ($(2.west)+(0,3pt)$);
    \draw[->] ($(1'.east)-(0,3pt)$) to ($(2'.west)-(0,3pt)$);
    \draw[->] ($(1'.east)+(0,3pt)$) to ($(2'.west)+(0,3pt)$);
    \draw[->,prosmall] (1) to (1');
    \draw[->,prosmall] (2) to (2');
    \node[la] at ($(1)!0.5!(2')-(-3pt,2pt)$) {\rotatebox{270}{$\Rightarrow$}};
    \node[la] at ($(1)!0.5!(2')+(-3pt,2pt)$) {\rotatebox{270}{$\Rightarrow$}};
    \node[right of=2](1) {$\boldsymbol{\cdot}$}; 
    \node at ($(1)!0.5!(2')$) {$\to$}; 
    \node[right of=1](2) {$\boldsymbol{\cdot}$};
    \node[below of=1](1') {$\boldsymbol{\cdot}$};
    \node[below of=2](2') {$\boldsymbol{\cdot}$};
    \draw[->] (1) to (2);
    \draw[->] (1') to (2');
    \draw[->,prosmall] (1) to (1');
    \draw[->,prosmall] (2) to (2');
    \node[la] at ($(1)!0.5!(2')$) {\rotatebox{270}{$\Rightarrow$}};
    \end{tikzpicture}} & identity \\ \hline 
    \end{tabularx}
    \vspace{.1cm}
    \caption{Generating cofibrations}
    \label{tab:gencof}
\end{table}

\begin{table}[ht]
    \centering
    \begin{tabularx}{0.52\textwidth}{c|c}
    $\Set^{\Dop} \backslash \cat$ & $\mathbbm{1}\to \bV\bI$ \\ \hline 
   $\emptyset\to \mathbbm{1}$ & \parbox[c]{1.43cm}{\begin{tikzpicture}[node distance=1cm]
    \node[](1) {$\boldsymbol{\cdot}$};
    \node[right of=1,xshift=.5cm,yshift=1cm](1') {$\boldsymbol{\cdot}$};
    \node[below of=1'](2) {$\boldsymbol{\cdot}$}; 
    \draw[->,prosmall] (1') to node(a)[above,lasmall,sloped]{$\cong$} (2);
    \node at ($(1)!0.5!(1')$) {$\to$};
    \end{tikzpicture}} \\ \hline
    $\mathbbm{1}\sqcup\mathbbm{1}\to \bH\mathbbm{2}$ & \parbox[c]{3.9cm}{\begin{tikzpicture}[node distance=1cm]
    \node[](1) {$\boldsymbol{\cdot}$}; 
    \node[right of=1](2) {$\boldsymbol{\cdot}$};
    \node[below of=1](1') {$\boldsymbol{\cdot}$};
    \node[below of=2](2') {$\boldsymbol{\cdot}$};
    \draw[->] (1') to (2');
    \draw[->,prosmall] (1) to node[below,lasmall,sloped]{$\cong$} (1');
    \draw[->,prosmall] (2) to node[above,lasmall,sloped]{$\cong$} (2');
    \node[right of=2,xshift=.5cm](1) {$\boldsymbol{\cdot}$}; 
    \node at ($(1)!0.5!(2')$) {$\to$}; 
    \node[right of=1](2) {$\boldsymbol{\cdot}$};
    \node[below of=1](1') {$\boldsymbol{\cdot}$};
    \node[below of=2](2') {$\boldsymbol{\cdot}$};
    \draw[->] (1) to (2);
    \draw[->] (1') to (2');
    \draw[->,prosmall] (1) to node[below,lasmall,sloped]{$\cong$} (1');
    \draw[->,prosmall] (2) to node[above,lasmall,sloped]{$\cong$} (2');
    \node[lasmall] at ($(1)!0.5!(2')$) {\rotatebox{270}{$\cong$}};
    \end{tikzpicture}}  \\ \hline
    $\bH(\boldsymbol{\cdot} \rightrightarrows \boldsymbol{\cdot})\to \bH\mathbbm{2}$ & \parbox[c]{3.9cm}{\begin{tikzpicture}[node distance=1cm]
    \node[](1) {$\boldsymbol{\cdot}$}; 
    \node[above of=1, yshift=-.8cm] {};
    \node[right of=1](2) {$\boldsymbol{\cdot}$};
    \node[below of=1](1') {$\boldsymbol{\cdot}$};
    \node[below of=2](2') {$\boldsymbol{\cdot}$};
    \draw[->] ($(1.east)-(0,3pt)$) to ($(2.west)-(0,3pt)$);
    \draw[->] ($(1.east)+(0,3pt)$) to ($(2.west)+(0,3pt)$);
    \draw[->] (1') to (2');
    \draw[->,prosmall] (1) to node[below,lasmall,sloped]{$\cong$} (1');
    \draw[->,prosmall] (2) to node[above,lasmall,sloped]{$\cong$} (2');
    \node[lasmall] at ($(1)!0.5!(2')-(-3.5pt,2pt)$) {\rotatebox{270}{$\cong$}};
    \node[lasmall] at ($(1)!0.5!(2')+(-3.5pt,2pt)$) {\rotatebox{270}{$\cong$}};
    \node[right of=2,xshift=.5cm](1) {$\boldsymbol{\cdot}$}; 
    \node at ($(1)!0.5!(2')$) {$\to$}; 
    \node[right of=1](2) {$\boldsymbol{\cdot}$};
    \node[below of=1](1') {$\boldsymbol{\cdot}$};
    \node[below of=2](2') {$\boldsymbol{\cdot}$};
    \draw[->] (1) to (2);
    \draw[->] (1') to (2');
    \draw[->,prosmall] (1) to node[below,lasmall,sloped]{$\cong$} (1');
    \draw[->,prosmall] (2) to node[above,lasmall,sloped]{$\cong$} (2');
    \node[lasmall] at ($(1)!0.5!(2')$) {\rotatebox{270}{$\cong$}};
    \end{tikzpicture}} \\ \hline
    \end{tabularx}
    \caption{Generating anodyne extensions}
    \label{tab:genanext}
\end{table}

\begin{rem} \label{lastgencof}
In \cref{tab:gencof}, we can replace the pushout-products of the double functors $\mathbbm{1}\sqcup\mathbbm{1}\to \bV\mathbbm{2}$ with $\bH(\boldsymbol{\cdot}\rightrightarrows\boldsymbol{\cdot})\to \bH\mathbbm{2}$ and of $\bV(\boldsymbol{\cdot}\rightrightarrows\boldsymbol{\cdot})\to \bV\mathbbm{2}$ with $\mathbbm{1}\sqcup\mathbbm{1}\to \bH\mathbbm{2}$ by just one double functor $\bV\mathbbm{2}\times \bH\mathbbm{2} \sqcup_{\partial(\bV\mathbbm{2}\times\bH\mathbbm{2})} \bV\mathbbm{2}\times \bH\mathbbm{2} \to \bV\mathbbm{2}\times \bH\mathbbm{2}$ which can be pictured as
\begin{tz}[node distance=1cm]
    \node[](1) {$\boldsymbol{\cdot}$}; 
    \node[right of=1](2) {$\boldsymbol{\cdot}$};
    \node[below of=1](1') {$\boldsymbol{\cdot}$};
    \node[below of=2](2') {$\boldsymbol{\cdot}$};
    \draw[->] (1) to (2);
    \draw[->] (1') to (2');
    \draw[->,prosmall] (1) to (1');
    \draw[->,prosmall] (2) to (2');
    \node[la] at ($(1)!0.5!(2')-(-3.5pt,0)$) {\rotatebox{270}{$\Rightarrow$}};
    \node[la] at ($(1)!0.5!(2')+(-3.5pt,0)$) {\rotatebox{270}{$\Rightarrow$}};
    \node[right of=2](1) {$\boldsymbol{\cdot}$}; 
    \node at ($(1)!0.5!(2')$) {$\to$}; 
    \node[right of=1](2) {$\boldsymbol{\cdot}$};
    \node[below of=1](1') {$\boldsymbol{\cdot}$};
    \node[below of=2](2') {$\boldsymbol{\cdot}$};
    \draw[->] (1) to (2);
    \draw[->] (1') to (2');
    \draw[->,prosmall] (1) to (1');
    \draw[->,prosmall] (2) to (2');
    \node[la] at ($(1)!0.5!(2')$) {\rotatebox{270}{$\Rightarrow$}};
    \end{tz}
Indeed, the above double functor can be seen as a pushout of either of the pushout-products, and the pushout-products can be seen as an iterated pushout of the double functors $\bV(\boldsymbol{\cdot}\rightrightarrows\boldsymbol{\cdot})\to\bV\mathbbm{2}$ and $\bH(\boldsymbol{\cdot}\rightrightarrows\boldsymbol{\cdot})\to\bH\mathbbm{2}$, respectively, as well as the above double functor. 
\end{rem}

By a similar argument to the proof of M4 in \cite[Theorem 3.1]{rezk}, we can use \cref{tab:gencof,lastgencof} to get the following characterization of the cofibrations in $\dblcat$.

\begin{prop} \label{charcofib}
A double functor is a cofibration if and only if it is injective on objects. In particular, every double category is cofibrant. 
\end{prop}

\subsection{Naive fibrant objects and trivial fibrations}

We now give characterizations of the naive fibrant objects and trivial fibrations. 

\begin{prop} \label{charnaivefibNh}
A double category $\bA$ is naive fibrant if and only if, for every diagram in $\bA$ as below left with $u$, $v$ vertical isomorphisms, there is a unique pair $(f,\alpha)$ of a horizontal morphism $f$ and a vertically invertible square $\alpha$ in $\bA$ as below right.
\begin{tz}
\node[](1) {$A$}; 
\node[below of=1](2) {$A'$}; 
\node[right of=1](3) {$C$}; 
\node[below of=3](4) {$C'$};

\draw[->,pro] (1) to node[left,la]{$u$} node[right,la]{$\cong$} (2); 
\draw[->] (2) to node[below,la]{$f'$} (4); 
\draw[->,pro](3) to node[right,la]{$v$} node[left,la]{$\cong$} (4);

\node[right of=3,xshift=1.5cm](1) {$A$}; 
\node[below of=1](2) {$A'$}; 
\node[right of=1](3) {$C$}; 
\node[right of=2](4) {$C'$};

\draw[->,pro] (1) to node[left,la]{$u$} (2); 
\draw[->] (1) to node[above,la]{$f$} (3); 
\draw[->] (2) to node[below,la]{$f'$} (4); 
\draw[->,pro](3) to node[right,la]{$v$} (4); 
 
\node[la] at ($(1)!0.5!(4)+(5pt,0)$) {$\alpha$};
\node[la] at ($(1)!0.5!(4)-(5pt,0)$) {\rotatebox{90}{$\cong$}};
\end{tz}
\end{prop}

\begin{proof}
This is obtained directly from a close inspection of the right lifting properties of~$\bA$ with
respect to the generating anodyne extensions of \cref{tab:genanext}. 
\end{proof}

\begin{rem}
By a dual of \cite[Theorem 4.1.7]{grandis}, the requirement that the pairs $(f,\alpha)$ as above exist is equivalent to asking that every vertical isomorphism in $\bA$ has a companion. However, the requirement that the pairs $(f,\alpha)$ are unique is stronger than requiring that the companions for the vertical isomorphisms are unique.
\end{rem}

We now characterize the trivial fibrations in two useful ways.

\begin{prop} \label{chartrivfib}
For a double functor $F\colon \bA\to \bB$, the following are equivalent.
\begin{rome}
   \item The double functor $F$ is a trivial fibration. 
   \item The double functor $F$ is surjective on objects and fully faithful on horizontal and vertical morphisms and squares.
   \item There exists a tuple $(G,\eta,\chi,\mu)$ consisting of
\begin{itemize}
   \item a double functor $G\colon \bB\to \bA$ such that $F G=\id_{\bB}$,
   \item a horizontal natural isomorphism $\eta\colon GF\cong \id_{\bA}$ such that $F\eta=\id_F$,
   \item a vertical natural isomorphism $\chi\colon G F\cong \id_{\bA}$ such that $F\chi=e_F$,
   \item a modification $\mu$ of the form 
\begin{tz}
\node[](1) {$GF$}; 
\node[below of=1](2) {$\id_\bA$}; 
\node[right of=1](3) {$\id_\bA$}; 
\node[below of=3](4) {$\id_\bA$};

\draw[n,pro] (1) to node[left,la]{$\chi$} (2); 
\draw[n] (1) to node[above,la]{$\eta$} (3); 
\draw[d] (2) to (4); 
\draw[d,pro](3) to (4); 
 
\node[la] at ($(1)!0.5!(4)$) {$\mu$};
\end{tz}
which is both horizontally and vertically invertible and such that $F\mu=\square_F$.
\end{itemize}
\end{rome}
\end{prop}

\begin{proof}
The fact that (i) and (ii) are equivalent is obtained directly from a close inspection of the right lifting properties of $F$ with
respect to the generating cofibrations of \cref{tab:gencof}, together with \cref{lastgencof}. We show that (ii) and (iii) are equivalent. 

Suppose that $F$ satisfies (ii). We define a double functor $G\colon \bB\to \bA$. Given an object $B\in \bB$, since $F$ is surjective on objects, there exists an object $A\in \bA$ such that $FA=B$ and we set $GB\coloneqq A$. Now, since $F$ is fully faithful on horizontal and vertical morphisms and squares, there is a unique way of extending $G$ on horizontal and vertical morphisms and squares in such a way that $FG=\id_{\bB}$. In particular, this defines a double functor by unicity of the lifts. 

Now we construct the horizontal natural isomorphism $\eta\colon GF\cong \id_{\bA}$. Given an object $A\in \bA$, by definition of $G$, we have that $FGFA=FA$. Since $F$ is fully faithful on horizontal morphisms, there exists a unique isomorphism $f\colon GFA\cong A$ in $\bA$ such that $Ff=\id_{FA}$ and we set $\eta_A\coloneqq f$. Then, given a vertical morphism $u\colon A\arrowdot A'$ in $\bA$, we have the following identity square as depicted below left. 
\begin{tz}
\node[](1) {$FGFA$}; 
\node[below of=1](2) {$FGFA'$}; 
\node[right of=1,xshift=.5cm](3) {$FA$}; 
\node[below of=3](4) {$FA'$};

\draw[->,pro] (1) to node[left,la]{$FGFu$} (2); 
\draw[d] (1) to node[above,la]{$F\eta_A$} (3); 
\draw[d] (2) to node[below,la]{$F\eta_{A'}$} (4); 
\draw[->,pro](3) to node[right,la]{$Fu$} (4); 
 
\node[la] at ($(1)!0.5!(4)$) {$\id_{Fu}$};

\node[right of=3,xshift=1.5cm](1) {$GFA$}; 
\node[below of=1](2) {$GFA'$}; 
\node[right of=1](3) {$A$}; 
\node[right of=2](4) {$A'$};

\draw[->,pro] (1) to node[left,la]{$GFu$} (2); 
\draw[->] (1) to node[above,la]{$\eta_A$} (3); 
\draw[->] (2) to node[below,la]{$\eta_{A'}$} (4); 
\draw[->,pro](3) to node[right,la]{$u$} (4); 
 
\node[la] at ($(1)!0.5!(4)+(5pt,0)$) {$\alpha$};
\node[la] at ($(1)!0.5!(4)-(5pt,0)$) {$\cong$};
\end{tz}
Since $F$ is fully faithful on squares, there is a unique horizontally invertible square $\alpha$ in $\bA$ as depicted above right such that $F\alpha=\id_{Fu}$ and we set $\eta_u\coloneqq \alpha$. Naturality of $\eta$ follows from fully faithfulness of $F$ on horizontal morphisms and squares and, by construction, we have that $F\eta=\id_F$. The construction of the vertical isomorphism $\chi\colon GF\cong \id_{\bA}$ with $F\chi=e_F$ is similar. 

Finally, we construct the modification $\mu$. Given an object $A\in \bA$, we have the following identity square as depicted below left. 
\begin{tz}
\node[](1) {$FGFA$}; 
\node[below of=1](2) {$FA$}; 
\node[right of=1,xshift=.5cm](3) {$FA$}; 
\node[below of=3](4) {$FA$};

\draw[d,pro] (1) to node[left,la]{$F\chi_A$} (2); 
\draw[d] (1) to node[above,la]{$F\eta_A$} (3); 
\draw[d] (2) to (4); 
\draw[d,pro](3) to (4); 
 
\node[la] at ($(1)!0.5!(4)$) {$\square_{FA}$};

\node[right of=3,xshift=1.5cm](1) {$GFA$}; 
\node[below of=1](2) {$A$}; 
\node[right of=1](3) {$A$}; 
\node[right of=2](4) {$A$};

\draw[->,pro] (1) to node[left,la]{$\chi_A$} (2); 
\draw[->] (1) to node[above,la]{$\eta_A$} (3); 
\draw[d] (2) to (4); 
\draw[d,pro](3) to (4); 
 
\node[la] at ($(1)!0.5!(4)$) {$\varphi$};
\end{tz}
Since $F$ is fully faithful on squares, there is a unique horizontally and vertically invertible square $\varphi$ as depicted above right such that $F\varphi=\square_{FA}$ and we set $\mu_A\coloneqq \varphi$. Compatibility of $\mu$ with $\eta$ and $\chi$ follows from fully faithfulness of $F$ on squares and, by construction, we have that $F\mu=\square_F$. 

Now suppose that $F$ admits a tuple $(G,\eta,\chi,\mu)$ as in (iii). We show that $F$ satisfies (ii). From the equality $FG=\id_{\bB}$ we directly get that $F$ is surjective on objects. Now given two objects $A,C\in \bA$ and a horizontal morphism $g\colon FA\to FC$ in $\bB$, then the composite of horizontal morphisms in $\bA$
\[ A\xrightarrow{\eta_A^{-1}} GFA\xrightarrow{Gg}GFC\xrightarrow{\eta_{C}} C \]
is mapped by $F$ to $g$ since $F\eta=\id_F$ and $FG=\id_{\bB}$. Hence $F$ is full on horizontal morphisms. If $f\colon A\to C$ is another horizontal morphism in $\bA$ such that $Ff=g$, by naturality of $\eta$ we have the following commutative square of horizontal morphisms in~$\bA$
\begin{tz}
\node[](1) {$GFA$}; 
\node[below of=1](2) {$GFC$}; 
\node[right of=1](3) {$A$}; 
\node[below of=3](4) {$C$};

\draw[->] (1) to node[left,la]{$GFf=Gg$} (2); 
\draw[->] (1) to node[above,la]{$\eta_A$} (3); 
\draw[->] (2) to node[below,la]{$\eta_{C}$} (4); 
\draw[->](3) to node[right,la]{$f$} (4); 
\end{tz}
so that $f=\eta_C(Gg)\eta_A^{-1}$. Hence $F$ is faithful on horizontal morphisms. Similarly, one can show using $\chi$ that $F$ is fully faithful on vertical morphisms. Finally, fully faithfulness on squares follows from a similar argument to the one for horizontal morphisms, using that $\eta$ has horizontally invertible square components and that it is natural with respect to squares.
\end{proof}

\subsection{Proof of \cref{thm:SRIalongnerve}}

To show that every trivial fibration is a weak equivalence, we first need the next technical lemmas, which allow us to compute a specific naive fibrant replacement of a trivial fibration which is also a trivial fibration. 

\begin{lem} \label{technicallemma}
Let $F\colon \bA\to \bB$ be a double functor that admits a tuple $(G,\eta,\chi,\mu)$ as in \cref{chartrivfib} (iii) and let $I\colon \bA\to \bC$ be a cofibration. Then there exists a tuple $(H,\theta,\psi,\nu)$ consisting of 
\begin{itemize}
   \item a double functor $H\colon \bC\to \bC$ such that $H I=I G F$, and $H C=C$ for every object $C\in \bC$ which is not in the image of $I$,
   \item a horizontal natural isomorphism $\theta\colon H\stackrel{\cong}{\Rightarrow} \id_{\bC}$ such that $\theta I=I\eta$, and $\theta_C=\id_C$ for every object $C\in \bC$ which is not in the image of $I$,
   \item a vertical natural isomorphism $\psi\colon H\Isodot \id_{\bC}$ such that $\psi I=I\chi$, and $\psi_C=e_C$ for every object $C\in \bC$ which is not in the image of $I$,
   \item a modification $\nu$ of the form 
\begin{tz}
\node[](1) {$H$}; 
\node[below of=1](2) {$\id_\bC$}; 
\node[right of=1](3) {$\id_\bC$}; 
\node[below of=3](4) {$\id_\bC$};

\draw[n,pro] (1) to node[left,la]{$\psi$} (2); 
\draw[n] (1) to node[above,la]{$\theta$} (3); 
\draw[d] (2) to (4); 
\draw[d,pro](3) to (4); 
 
\node[la] at ($(1)!0.5!(4)$) {$\nu$};
\end{tz}
which is both horizontally and vertically invertible and such that $\nu I=I\mu$, and $\nu_C=\square_C$ for every object $C\in \bC$ which is not in the image of $I$.
\end{itemize}
\end{lem}

\begin{proof}
Since $I$ is injective on objects by \cref{charcofib}, given an object $C\in \bC$, we have that either $C=IA$ for a unique object $A\in \bA$ or that $C$ is not in the image of $I$. Hence, the statement determines $(H,\theta,\psi,\nu)$ on objects. The rest of the data is then completely determined by the required naturality of $\theta$, $\psi$, and $\nu$. 
\end{proof}

\begin{lem} \label{lem:pushoftrivfib}
The pushout of a trivial fibration along a cofibration is a trivial fibration.
\end{lem}

\begin{proof}
Let $F\colon \bA\to \bB$ be a trivial fibration and $I\colon \bA\to \bC$ be a cofibration. We want to show that the double functor $F'\colon \bC\to \bP$ in the following pushout is a trivial fibration.
\begin{tz}
\node[](1) {$\bA$}; 
\node[right of=1](2) {$\bC$}; 
\node[below of=1](1') {$\bB$};
\node[below of=2](2') {$\bP$}; 

\draw[->] (1) to node[above,la]{$I$} (2); 
\draw[->] (1) to node[left,la]{$F$} (1'); 
\draw[->] (2) to node[right,la]{$F'$} (2');
\draw[->] (1') to node[below,la]{$I'$} (2');
\node at ($(2')-(.3cm,-.3cm)$) {$\ulcorner$};
\end{tz}
By (iii) of \cref{chartrivfib}, there exists a tuple $(G,\eta,\chi,\mu)$ for $F$. By \cref{technicallemma}, we then get a tuple $(H,\theta,\psi,\nu)$ compatible with $I$ and $(G,\eta,\chi,\mu)$. By the universal property of the pushout, using that $HI=IGF$, there exists a unique double functor $G'\colon \bP\to \bC$ making the following diagram commute. 
\begin{tz}
\node[](1) {$\bA$}; 
\node[right of=1](2) {$\bC$}; 
\node[below of=1](1') {$\bB$};
\node[below of=2](2') {$\bP$}; 
\node[below right of=1',xshift=.3cm](3') {$\bA$};
\node[right of=3'](3) {$\bC$};

\draw[->] (1) to node[above,la]{$I$} (2); 
\draw[->] (1) to node[left,la]{$F$} (1'); 
\draw[->] (2) to node[right,la]{$F'$} (2');
\draw[->] (1') to node[above,la]{$I'$} (2');
\node at ($(2')-(.3cm,-.3cm)$) {$\ulcorner$};
\draw[->,bend left] (2) to node[right,la]{$H$} (3);
\draw[->] (1') to node[below,la,xshift=-2pt,pos=0.4]{$G$} (3');
\draw[->] (3') to node[below,la]{$I$} (3);
\draw[->,dashed] (2') to node[above,la,xshift=2pt,pos=0.6]{$G'$} (3);
\end{tz}
We now show that $F'H=F'$ by showing that $F'\psi\colon F'H\Isodot F'$ is the identity at $F'$. Given an object $C\in \bC$, if $C=IA$ for a unique $A\in \bA$, then using that $\psi I=I\chi$ and that $F\chi=e_F$ we get that
\[ F'\psi_{IA}=F'I\chi_A=I'F\chi_A=I'e_{FA}=e_{I'FA}=e_{F'IA}. \]
Similarly, since $\nu I=I\mu$ and $F\mu=\square_F$, we have $F'\nu_{IA}=\square_{F'IA}$. Otherwise, if $C$ is not in the image of $I$, then using that $\psi_C=e_C$ we get that $F'\psi_C=F'e_C=e_{F'C}$. Similarly, since $\nu_C=\square_C$, we have $F'\nu_C=\square_{F'C}$. Now, if $f\colon C\to D$ is a horizontal morphism in $\bC$, by naturality of $\nu$, we have \[ F'\psi_f=(F'\nu_D^{-1,h})\circ_h F'e_f\circ_h (F'\nu_C)=\square_{F'D}\circ_h e_{F'f}\circ_h \square_{F'C}=e_{F'f}. \]
This shows that $F'\psi=e_{F'}$ and so $F'H=F'$. Similarly, one can show that $F'\theta=\id_{F'}$. 

Using that $F'H=F'$ and $FG=\id_{\bB}$, by the universal property of the pushout we have that $F'G'=\id_{\bP}$, and since $G'F'=H$, the tuple $(G',\theta,\psi,\nu)$ gives a tuple as in \cref{chartrivfib} (iii) for $F'$. This shows that $F'$ is a trivial fibration.
\end{proof}

\begin{lem} \label{trivfibfromfibrant}
Let $F\colon \bA\to \bB$ be a trivial fibration such that $\bA$ is naive fibrant. Then $\bB$ is also naive fibrant. 
\end{lem}

\begin{proof}
To prove this lemma, we use the characterization of naive fibrant objects from \cref{charnaivefibNh}. Suppose that we are given a diagram in $\bB$ as below left.
\begin{tz}
\node[](1) {$B$}; 
\node[below of=1](2) {$B'$}; 
\node[right of=1](3) {$D$}; 
\node[below of=3](4) {$D'$};

\draw[->,pro] (1) to node[left,la]{$v$} node[right,la]{\rotatebox{270}{$\cong$}} (2); 
\draw[->] (2) to node[below,la]{$g'$} (4); 
\draw[->,pro](3) to node[right,la]{$v'$} node[left,la]{\rotatebox{270}{$\cong$}} (4); 

\node[right of=3,xshift=1.5cm](1) {$A$}; 
\node[below of=1](2) {$A'$}; 
\node[right of=1](3) {$C$}; 
\node[right of=2](4) {$C'$};

\draw[->,pro] (1) to node[left,la]{$v$} node[right,la]{\rotatebox{270}{$\cong$}} (2); 
\draw[->] (2) to node[below,la]{$f'$} (4);  
\draw[->,pro](3) to node[right,la]{$u'$} node[left,la]{\rotatebox{270}{$\cong$}} (4); 
\end{tz}
Since $F$ is surjective on objects, let us fix objects $A,A',C,C'\in \bA$ that map under $F$ to $B,B',D,D'\in\bB$ respectively. Since $F$ is fully faithful on horizontal and vertical morphisms, there are unique vertical isomorphism $u,u'$ in $\bA$ and a unique horizontal morphism $f'$ in $\bA$ as depicted above right such that $Fu=v$, $Fu'=v'$, and $Ff'=f$. Since $\bA$ is naive fibrant, there is a unique pair $(f,\alpha)$ of a horizontal morphism $f$ and a vertically invertible square $\alpha$ in $\bA$ as below left, 
\begin{tz}
\node[](1) {$A$}; 
\node[below of=1](2) {$A'$}; 
\node[right of=1](3) {$C$}; 
\node[right of=2](4) {$C'$};

\draw[->,pro] (1) to node[left,la]{$u$} (2); 
\draw[->] (1) to node[above,la]{$f$} (3); 
\draw[->] (2) to node[below,la]{$f'$} (4); 
\draw[->,pro](3) to node[right,la]{$v$} (4); 
 
\node[la] at ($(1)!0.5!(4)-(5pt,0)$) {$\alpha$};
\node[la] at ($(1)!0.5!(4)+(5pt,0)$) {\rotatebox{270}{$\cong$}};

\node[right of=3,xshift=1.5cm](1) {$B$}; 
\node[below of=1](2) {$B'$}; 
\node[right of=1](3) {$D$}; 
\node[right of=2](4) {$D'$};

\draw[->,pro] (1) to node[left,la]{$v$} (2); 
\draw[->] (1) to node[above,la]{$Ff$} (3); 
\draw[->] (2) to node[below,la]{$g'$} (4); 
\draw[->,pro](3) to node[right,la]{$v'$} (4); 
 
\node[la] at ($(1)!0.5!(4)-(6pt,0)$) {$F\alpha$};
\node[la] at ($(1)!0.5!(4)+(6pt,0)$) {\rotatebox{270}{$\cong$}};
\end{tz}
so its image gives a pair $(Ff,F\alpha)$ of a horizontal morphism $Ff$ and a vertically invertible square $F\alpha$ as depicted above right. This shows the existence of such a pair. For the unicity, since $F$ is fully faithful on horizontal morphisms and squares, it lifts such a pair uniquely (once we have fixed the objects $A,A',C,C'$) and so the pair $(Ff,\alpha)$ must be unique. This shows that $\bB$ is naive fibrant.
\end{proof}

Using the above results, we can now prove \ref{srind2} of \cref{thm:semirightinduced}.

\begin{prop} \label{trivfibarewedblcat}
Every trivial fibration is a weak equivalence.
\end{prop}

\begin{proof}
Let $F\colon \bA\to \bB$ be a trivial fibration. Choose a fibrant replacement $\bA\to \bA^\fibrant$ and consider the following pushout in $\dblcat$.
\begin{tz}
\node[](1) {$\bA$}; 
\node[right of=1](2) {$\bA^\fibrant$}; 
\node[below of=1](1') {$\bB$};
\node[below of=2](2') {$\bP$}; 

\draw[->] (1) to (2); 
\draw[->] (1) to node[left,la]{$F$} (1'); 
\draw[->] (2) to node[right,la]{$F'$} (2');
\draw[->] (1') to (2');
\node at ($(2')-(.3cm,-.3cm)$) {$\ulcorner$};
\end{tz}
Since $\bA\to \bA^\fibrant$ is an anodyne extension, it is in particular a cofibration. Then, by \cref{lem:pushoftrivfib}, we have that the double functor $F'\colon \bA^\fibrant \to \bP$ is a trivial fibration. Since $\bA^\fibrant$ is naive fibrant, we get by \cref{trivfibfromfibrant} that $\bP$ is also naive fibrant. Moreover, the double functor $\bB\to \bP$ is an anodyne extension as a pushout of such. Hence the above square is a naive fibrant replacement of $F$. Since $\bA^\fibrant\to \bP$ is a trivial fibration, it is such that its horizontal nerve $N^h\bA^\fibrant\to N^h\bP$ is a trivial fibration and so a weak equivalence in $\cat^{\Dop}$. This shows that $F$ is a weak equivalence.
\end{proof}

It remains to show that there is a path object for every naive fibrant double category. We denote by $[-,-]$ the internal homs in $\cat$ and by $[-,-]_\Delta$ the internal homs in $\cat^{\Dop}$. 

\begin{lem} \label{nervevsinthom}
Let $\cC\in \cat\subseteq\cat^{\Dop}$ and $\bA\in \dblcat$. Then there is a natural isomorphism in $\cat^{\Dop}$
\[ [\cC,N^h\bA]_\Delta\cong N^h\llbracket \bV\cC,\bA\rrbracket. \]
\end{lem}

\begin{proof}
For every $n\geq 0$, we have the following natural isomorphisms of categories
\begin{align*}
    ([\cC,N^h\bA]_\Delta)_n &= [\cC, (N^h\bA)_n] & \text{def.\ of enr.\ hom} \\
    &\cong [\cC, \bfV\llbracket\bH [n], \bA\rrbracket ] & \text{def.\ of } N^h \\
    & \cong \bfV\llbracket\bV\cC, \llbracket\bH [n],\bA\rrbracket \rrbracket& \bV\dashv\bfV \text{ enriched\footnotemark} \\
    & \cong \bfV\llbracket\bH [n]\times \bV\cC,\bA\rrbracket & \text{internal hom} \\
    &\cong \bfV\llbracket\bH [n], \llbracket\bV\cC,\bA\rrbracket \rrbracket & \text{internal hom} \\
    & \cong (N^h \llbracket\bV\cC,\bA\rrbracket)_n & \text{def.\ of } N^h
\end{align*}
\footnotetext{This is shown in \cite[Proposition 2.5]{FPP}.}
which yield a natural isomorphism $[\cC,N^h\bA]_\Delta\cong N^h\llbracket\bV\cC,\bA\rrbracket$ in $\cat^{\Dop}$. 
\end{proof}

\begin{prop} \label{prop:pathobjdblcat}
Let $\bA$ be a naive fibrant double category. Then the factorization of the diagonal at $\bA$
\[ \bA\to \llbracket\bV\bI,\bA\rrbracket\to \bA\times \bA \]
induced by $\mathbbm 1\sqcup \mathbbm 1\to \bV\bI \to \mathbbm 1$ is a path object. 
\end{prop}

\begin{proof}
By applying the nerve $N^h$ to the factorization $\bA\to \llbracket\bV\bI,\bA\rrbracket\to \bA\times \bA$, we get by \cref{nervevsinthom} the following factorization of the diagonal at $N^h\bA$
\[ N^h\bA\to [\bV\bI,N^h\bA]_\Delta\to N^h\bA\times N^h\bA. \]
Since the model structure on $\cat$ is monoidal, we have that the Reedy model structure on $\cat^{\Dop}$ is enriched over $\cat$ by \cite[Lemma 4.2]{BarwickReedy}. Using the facts that $N^h\bA$ is fibrant in $\cat^{\Dop}$ and that $\mathbbm{1}\sqcup\mathbbm{1}\to \bI$ is a cofibration and $\mathbbm{1}\to \bI$ a trivial cofibration in $\cat$, it then follows that the first morphism in the above factorization is a weak equivalence and the second one a fibration in $\cat^{\Dop}$, as desired.
\end{proof}

\begin{proof}[Proof of \cref{thm:SRIalongnerve}]
The existence of the fibrantly-transferred model structure is given by \cref{thm:semirightinduced} using \cref{trivfibarewedblcat,prop:pathobjdblcat}.
\end{proof}

Finally, we can extract an explicit description of the weak equivalences between fibrant objects in this new model structure.

\begin{prop} \label{prop:webtwnaivefib}
A double functor $F\colon \bA\to \bB$ between fibrant objects is a weak equivalence if and only if it is vertically essentially surjective on objects, and fully faithful on horizontal and vertical morphisms and squares.
\end{prop}

\begin{proof}
    By construction, a double functor $F\colon \bA\to \bB$ between fibrant objects is a weak equivalence if and only if $N^h F$ is a weak equivalence in $\cat^{\Dop}$ equipped with the Reedy model structure; that is, if and only if, for all $n\geq 0$, the functor $(N^h F)_n$ is an equivalence of categories. As described in \cref{defn:hornerve}, we have that $(N^h\bA)_0=\bA_0$ is the category of objects and vertical morphisms of $\bA$, and so we see that $(N^h F)_0$ is an equivalence of categories if and only if $F$ is vertically essentially surjective on objects and fully faithful on vertical morphisms.

    Next, recall that $(N^h\bA)_1=\bA_1$ is the category of horizontal morphisms and squares of~$\bA$, hence the functor $(N^h F)_1$ is an equivalence of categories if and only if the following two conditions are satisfied:
    \begin{enumerate}
        \item for each horizontal morphism $g\colon B\to D$ in $\bB$, there exists a horizontal morphism $f\colon A\to C$ in $\bA$ and a vertically invertible square $\beta$ in $\bB$ as follows,
        \begin{tz}
\node[](1) {$B$}; 
\node[below of=1](2) {$FA$}; 
\node[right of=1](3) {$D$}; 
\node[right of=2](4) {$FC$};

\draw[->,pro] (1) to node[left,la]{$w$} (2); 
\draw[->] (1) to node[above,la]{$g$} (3); 
\draw[->] (2) to node[below,la]{$Ff$} (4); 
\draw[->,pro](3) to node[right,la]{$w'$} (4); 
 
\node[la] at ($(1)!0.5!(4)+(5pt,0)$) {$\beta$};
\node[la] at ($(1)!0.5!(4)-(5pt,0)$) {\rotatebox{90}{$\cong$}};
\end{tz}
    \item for each square $\beta$ in $\bB$ as below left, there exists a unique square $\alpha$ in $\bA$ as below right such that $F\alpha=\beta$.
    \begin{tz}
\node[](1) {$FA$}; 
\node[below of=1](2) {$FA'$}; 
\node[right of=1](3) {$FC$}; 
\node[below of=3](4) {$FC'$};

\draw[->,pro] (1) to node[left,la]{$v$}  (2); 
\draw[->] (2) to node[below,la]{$Ff'$} (4); 
\draw[->,pro](3) to node[right,la]{$v'$}  (4); 
\draw[->] (1) to node[above,la]{$Ff$} (3); 

\node[la] at ($(1)!0.5!(4)$) {$\beta$};

\node[right of=3,xshift=1.5cm](1) {$A$}; 
\node[below of=1](2) {$A'$}; 
\node[right of=1](3) {$C$}; 
\node[right of=2](4) {$C'$};

\draw[->,pro] (1) to node[left,la]{$u$} (2); 
\draw[->] (1) to node[above,la]{$f$} (3); 
\draw[->] (2) to node[below,la]{$f'$} (4); 
\draw[->,pro](3) to node[right,la]{$u'$} (4); 
 
\node[la] at ($(1)!0.5!(4)$) {$\alpha$};
\end{tz}
 \end{enumerate}

 Suppose that the functor $(N^hF)_n$ is an equivalence of categories, for all $n\geq 0$. Combining the fact that $F$ is fully faithful on vertical morphisms with condition (2) above, we obtain that $F$ is fully faithful on squares. Thus, to prove the desired description, it remains to show that $F$ is fully faithful on horizontal morphisms. For faithfulness, let $f,f'\colon A\to C$ be horizontal morphisms in $\bA$ such that $Ff=Ff'$, and consider the square in $\bB$ as below left.
         \begin{tz}
\node[](1) {$FA$}; 
\node[below of=1](2) {$FA$}; 
\node[right of=1](3) {$FC$}; 
\node[right of=2](4) {$FC$};

\draw[d,pro] (1) to  (2); 
\draw[->] (1) to node[above,la]{$Ff$} (3); 
\draw[->] (2) to node[below,la]{$Ff'$} (4); 
\draw[d,pro](3) to  (4); 
 
\node[la] at ($(1)!0.5!(4)$) {$e_{Ff}$};

\node[right of=3,xshift=1.5cm](1) {$A$}; 
\node[below of=1](2) {$A$}; 
\node[right of=1](3) {$C$}; 
\node[right of=2](4) {$C$};

\draw[d,pro] (1) to  (2); 
\draw[->] (1) to node[above,la]{$f$} (3); 
\draw[->] (2) to node[below,la]{$f'$} (4); 
\draw[d,pro](3) to  (4); 
 
\node[la] at ($(1)!0.5!(4)+(5pt,0)$) {$\alpha$};
\node[la] at ($(1)!0.5!(4)-(5pt,0)$) {\rotatebox{90}{$\cong$}};
\end{tz}
As $F$ is fully faithful on squares, there exists a unique square $\alpha$ in $\bA$ as above right which (again by fully faithfulness of $F$ on squares) must be vertically invertible. Then both $\alpha$ and $e_{f'}$ are vertically invertible squares in $\bA$ that complete the diagram 
         \begin{tz}
\node[](1) {$A$}; 
\node[below of=1](2) {$A$}; 
\node[right of=1](3) {$C$}; 
\node[right of=2](4) {$C$};

\draw[d,pro] (1) to  (2); 
\draw[->] (2) to node[below,la]{$f'$} (4); 
\draw[d,pro](3) to  (4); 
\end{tz} and, since $\bA$ is fibrant, we must have $\alpha=e_{f'}$; in particular, we have $f=f'$ as desired.

To show that $F$ is full on horizontal morphisms, suppose that we are given a horizontal morphism $g\colon FA\to FC$ in $\bB$. By condition (1) above, there is a vertically invertible square $\beta$ in $\bB$ as follows.
        \begin{tz}
\node[](1) {$FA$}; 
\node[below of=1](2) {$FA'$}; 
\node[right of=1](3) {$FC$}; 
\node[right of=2](4) {$FC'$};

\draw[->,pro] (1) to node[left,la]{$w$} (2); 
\draw[->] (1) to node[above,la]{$g$} (3); 
\draw[->] (2) to node[below,la]{$Ff$} (4); 
\draw[->,pro](3) to node[right,la]{$w'$} (4); 
 
\node[la] at ($(1)!0.5!(4)+(5pt,0)$) {$\beta$};
\node[la] at ($(1)!0.5!(4)-(5pt,0)$) {\rotatebox{90}{$\cong$}};
\end{tz} In particular, note that $w$ and $w'$ must be vertical isomorphisms. Now, since $F$ is fully faithful on vertical morphisms, there exist unique vertical isomorphisms $u\colon A\to A'$ and $u'\colon C\to C'$ in $\bA$ such that $Fu=w$, $Fu'=w'$. We can then consider the diagram below left, which can be completed to a unique vertically invertible square as below right using the fact that $\bA$ is fibrant. 
\begin{tz}
\node[](1) {$A$}; 
\node[below of=1](2) {$A'$}; 
\node[right of=1](3) {$C$}; 
\node[below of=3](4) {$C'$};

\draw[->,pro] (1) to node[left,la]{$u$} node[right,la]{$\cong$} (2); 
\draw[->] (2) to node[below,la]{$f$} (4); 
\draw[->,pro](3) to node[right,la]{$u'$} node[left,la]{$\cong$} (4);

\node[right of=3,xshift=1.5cm](1) {$A$}; 
\node[below of=1](2) {$A'$}; 
\node[right of=1](3) {$C$}; 
\node[right of=2](4) {$C'$};

\draw[->,pro] (1) to node[left,la]{$u$} (2); 
\draw[->] (1) to node[above,la]{$\hat{f}$} (3); 
\draw[->] (2) to node[below,la]{$f$} (4); 
\draw[->,pro](3) to node[right,la]{$u'$} (4); 
 
\node[la] at ($(1)!0.5!(4)+(5pt,0)$) {$\alpha$};
\node[la] at ($(1)!0.5!(4)-(5pt,0)$) {\rotatebox{90}{$\cong$}};
\end{tz} Applying $F$, we see that both $F\alpha$ and $\beta$ are vertically invertible squares in $\bB$ completing the diagram 
\begin{tz}
\node[](1) {$FA$}; 
\node[below of=1](2) {$FA'$}; 
\node[right of=1](3) {$FC$}; 
\node[below of=3](4) {$FC'$};

\draw[->,pro] (1) to node[left,la]{$w$} node[right,la]{$\cong$} (2); 
\draw[->] (2) to node[below,la]{$Ff$} (4); 
\draw[->,pro](3) to node[right,la]{$w'$} node[left,la]{$\cong$} (4); 
\end{tz} which, since $\bB$ is fibrant, implies that $F\alpha=\beta$. In particular, we have that $g=F\hat{f}$, showing that $F$ is full on horizontal morphisms.

Conversely, suppose that $F$ is vertically essentially surjective on objects, and fully faithful on
horizontal and vertical morphisms and squares. As we already explained, this immediately implies that $(N^h F)_0$ is an equivalence of categories. To show that $(N^h F)_1$ is an equivalence, note that the fact that $F$ is fully faithful on vertical morphisms and squares implies condition (2) above; i.e., that $(N^h F)_1$ is fully faithful on morphisms. To prove that condition (1) is also satisfied, let $g\colon B\to D$ be a horizontal morphism in $\bB$. Since $F$ is vertically essentially surjective on objects, there exist objects $A,C$ in $\bA$ and vertical isomorphisms $v,v'$ in $\bB$ yielding a diagram as below left.
\begin{tz}
\node[](1) {$FA$}; 
\node[below of=1](2) {$B$}; 
\node[right of=1](3) {$FC$}; 
\node[below of=3](4) {$D$};

\draw[->,pro] (1) to node[left,la]{$v$} node[right,la]{$\cong$} (2); 
\draw[->] (2) to node[below,la]{$g$} (4); 
\draw[->,pro](3) to node[right,la]{$v'$} node[left,la]{$\cong$} (4);

\node[right of=3,xshift=1.5cm](1) {$FA$}; 
\node[below of=1](2) {$B$}; 
\node[right of=1](3) {$FC$}; 
\node[right of=2](4) {$D$};

\draw[->,pro] (1) to node[left,la]{$v$} (2); 
\draw[->] (1) to node[above,la]{$\hat{g}$} (3); 
\draw[->] (2) to node[below,la]{$g$} (4); 
\draw[->,pro](3) to node[right,la]{$v'$} (4); 
 
\node[la] at ($(1)!0.5!(4)+(5pt,0)$) {$\beta$};
\node[la] at ($(1)!0.5!(4)-(5pt,0)$) {\rotatebox{90}{$\cong$}};
\end{tz} Since $\bB$ is fibrant, there exists a unique vertically invertible square $\beta$ as above right. Now, as $F$ is fully faithful on horizontal morphisms, we must have $\hat{g}=Ff$ for some unique horizontal morphism $f\colon A\to C$ in $\bA$; the square $\beta^{-1}$ then verifies condition (1). These ideas can be iterated to prove that $(N^h F)_n$ is an equivalence of categories for all $n\geq 2$ as well.
\end{proof}

\section{An overview of further applications}\label{section:newexamples}

In this final section, we give an overview of new model structures that have recently been constructed using the tools introduced in this paper. In particular, the ideas used to verify the hypotheses of the theorems share common features across these different examples, and a reader hoping to apply our results to their own setting may wish to consult these for inspiration.

\subsection{Stable model structure on symmetric spectra}

Consider the free-forgetful adjunction $U\colon \mathrm{Sp}^\Sigma\rightleftarrows \mathrm{Sp}_{\mathrm{st}}^{\mathbb{N}}\colon F$ between the category of symmetric spectra and the category of sequential spectra. It is well-known that the stable model structure on $\mathrm{Sp}^\Sigma$ cannot be right-transferred along this adjunction from the stable model structure on sequential spectra. Indeed, doing so would define the weak equivalences as the maps inducing isomorphisms on homotopy groups, rather than the maps inducing isomorphisms on generalized cohomology, and as a consequence the adjunction $F\dashv U$ would not be a Quillen equivalence.

In recent work \cite{symsp}, Malkiewich and Sarazola use \cref{thm:semirightinduced} to prove that the stable model structure on $\mathrm{Sp}^\Sigma$ can be fibrantly-transferred from sequential spectra along the forgetful functor. Notably, this makes it possible to construct this model structure while completely avoiding the technical notion of stable equivalences between non-fibrant objects, and only relying on well-known facts about $\pi_*$-isomorphisms in~$\mathrm{Sp}^{\mathbb{N}}$.

\subsection{Model structure for Segal spaces}

In \cite{MoserNuiten}, Moser and Nuiten use \cref{thm:main} to construct a model structure on the category of simplicial spaces in which the fibrant objects are the Segal spaces and the weak equivalences between fibrant objects are the Dwyer–Kan equivalences. They prove that this provides a model of $(\infty,1)$-categories sitting somewhere between Segal categories and complete Segal spaces, and that it retains several desirable properties found in other models such as being left proper and cartesian closed. Moreover, they show that the classical nerve functor from the canonical model structure on $\cat$ is right Quillen, a feature that is not present on the model structure for complete Segal spaces.

\subsection{Model structures for discrete and Grothendieck fibrations}

Moser and Sarazola used \cref{thm:main} to construct two model structures on the slice over a fixed category whose fibrant objects capture the notions of discrete fibrations and of Grothendieck fibrations \cite{grothMS}. Grothendieck fibrations are the output of the Grothendieck construction, an equivalence of categories that plays a central role in category theory, and that requires the 2-categorical notion of pseudofunctors. This work provides a new perspective on this classical topic, by showing that one can work with ordinary functors instead, if one considers the Grothendieck construction as a Quillen equivalence. 

\subsection{Model structures on double categories}

In recent work \cite{dblcatequivs}, the second, third, and fourth named authors show how any (combinatorial) model structure on the category $\dblcat$ whose trivial fibrations are the canonical ones---that is, the double functors which are surjective on objects, full on horizontal and vertical morphisms, and fully faithful on squares---can be constructed using \cref{thm:main}. 

In particular, they use \cref{thm:main} to efficiently recover different model structures present in the literature:   
\begin{itemize}[leftmargin=0.8cm]
     \item  The gregarious model structure of Campbell \cite{Camp}; this is the ``canonical'' model structure for double categories, and is initial among model structures with the canonical trivial fibrations, in the sense that every other such model structure is a localization of it.
     \item The model structure for weakly horizontally invariant double categories of Moser--Sarazola--Verdugo \cite{whi}; this was originally constructed as a model structure compatible with the horizontal inclusion $\bH\colon\twocat\to\dblcat$, and was used by Moser to define a nerve functor from double categories to double $(\infty, 1)$-categories \cite{lyne}. 
     \item The model structure for  equipments of Verdugo \cite{pauthesis}; this was used by Verdugo to prove a result on equivalence invariance of formal category theory.
 \end{itemize}
In addition, \cref{thm:main} is used to produce several new model structures whose homotopy theories encode a range of 2-dimensional structures:
 \begin{itemize}[leftmargin=0.8cm]
 \item A model structure for transposable double categories; this models the homotopy theory of 2-categories, as the square functor $\Sq\colon\twocat\to\dblcat$ is a Quillen equivalence. In \cite[Chapter 10, Theorem 5.2.3]{roz},  Gaitsgory\textendash Rozenblyum conjecture that the $\infty$-analogue of this functor is fully faithful and identify its essential image. Our model structure settles the conjecture in the strict case; the original conjecture has since been proved by Abell\'an \cite{Abellan}.
  \item Model structures whose homotopy theories model $2$-categories whose morphisms all have left (resp.\ both left and right) adjoints.
     \item A model structure for transposable double groupoids; this models the homotopy theory of $2$-groupoids.
      \item A model structure for either empty or contractible double categories; this models homotopy $(-1)$-types.
     \item A model structure for contractible double categories; this models homotopy $(-2)$-types.%, and is terminal among model structures with the canonical trivial fibrations, in the sense that it is a localization of every other such model structure.
 \end{itemize}

\bibliographystyle{alpha}
\bibliography{reference}

\end{document}